\documentclass[a4paper,twoside,10pt]{article}
\usepackage[a4paper,left=3cm,right=3cm, top=3cm, bottom=3cm]{geometry}
\usepackage[utf8]{inputenc}
\usepackage{mathrsfs}
\usepackage{amsmath}
\usepackage{amsthm}
\usepackage{amssymb}
\usepackage{cite}
\usepackage{color}
\usepackage{appendix}
\usepackage[author={Lorenzo}]{pdfcomment}
\theoremstyle{plain}
\newtheorem{theorem}{Theorem}[section]
\newtheorem{corollary}[theorem]{Corollary}
\newtheorem{lemma}[theorem]{Lemma}
\newtheorem{proposition}[theorem]{Proposition}
\theoremstyle{remark}
\newtheorem{remark}{Remark}

\DeclareTextCompositeCommand{\u}{T1}{i}{\u\imath}

\let\div\relax
\DeclareMathOperator{\div}{\nabla\cdot}

\DeclareMathOperator{\divbf}{\boldsymbol \nabla \cdot}
\DeclareMathOperator{\gradbf}{\bf \nabla}

\newcommand{\taubold}{\boldsymbol \tau}

\newcommand{\Nbb}{\mathbb N}
\newcommand{\Pbb}{\mathbb P}
\newcommand{\Rbb}{\mathbb R}
\newcommand{\RM}{\mathbb{RM}}

\newcommand{\qbf}{\mathbf q}

\newcommand{\nbf}{\mathbf n}

\newcommand{\Norm}[1]{{\left\|{#1} \right\|}}
\newcommand{\SemiNorm}[1]{{\left|{#1} \right|}}
\newcommand{\Normth}[1]{{\left\vert\kern-0.25ex\left\vert\kern-0.25ex\left\vert #1 
    \right\vert\kern-0.25ex\right\vert\kern-0.25ex\right\vert}}
\newcommand{\PiRM}{\boldsymbol \Pi_{\RM}}

\newcommand{\qbfRM}{\qbf^{\text{RM}}}
\newcommand{\Sigmabold}{\boldsymbol\Sigma}
\newcommand{\Sigmaboldtilde}{\widetilde{\Sigmabold}}

\newcommand{\ubf}{\mathbf u}
\newcommand{\vbf}{\mathbf v}
\newcommand{\cbf}{\mathbf c}
\newcommand{\Vbf}{\mathbf V}

\newcommand{\etabold}{\boldsymbol \eta}
\newcommand{\CBAz}{C_{\mathcal{BA},0}}
\newcommand{\CBAo}{C_{\mathcal{BA},1}}
\newcommand{\CBAoA}{\CBAo^A}
\newcommand{\CBAoB}{\CBAo^B}
\newcommand{\betaz}{\beta_0}
\newcommand{\CNLz}{C_{\mathcal{NL},0}}
\newcommand{\CNLo}{C_{\mathcal{NL},1}}
\newcommand{\CNLd}{C_{\mathcal{NL},d}}
\newcommand{\CNLzstar}{C_{\mathcal{NL},0}^*}
\newcommand{\betao}{\beta_1}
\newcommand{\Piz}{\Pi^0}
\newcommand{\Piboldzz}{\boldsymbol\Pi^0}
\newcommand{\nablaS}{\nabla_S}
\newcommand{\nablaSS}{\nabla_{SS}}
\newcommand{\betastarz}{\beta^*_0}
\newcommand{\Ccal}{\mathcal C}
\newcommand{\Ccalzinf}{\Ccal_0^\infty}
\DeclareMathOperator{\ds}{\mathrm{d}s}
\DeclareMathOperator{\dx}{\mathrm{d}x}
\DeclareMathOperator{\dy}{\mathrm{d}y}
\DeclareMathOperator{\dz}{\mathrm{d}z}

\DeclareMathOperator{\dt}{\mathrm{d}t}

\DeclareMathOperator{\dxi}{\mathrm{d}\xi}
\DeclareMathOperator{\deta}{\mathrm{d}\eta}

\newcommand{\Brho}{B_\rho}
\DeclareMathOperator{\supp}{supp}
\newcommand{\hOmega}{R}
\newcommand{\partialxj}{\partial_{x_j}}

\newcommand{\partialxjj}{\partial^2_{x_{j}^2}}

\newcommand{\partialxjl}{\partial^2_{x_{j}x_\ell}}
\newcommand{\partialxjjl}{\partial^3_{x_{j}^2x_\ell}}

\newcommand{\partialxl}{\partial_{x_\ell}}
\newcommand{\partialxil}{\partial_{\xi_\ell}}

\newcommand{\fhat}{\widehat f}
\newcommand{\ghat}{\widehat g}
\newcommand{\Rbbn}{\Rbb^n}
\newcommand{\xio}{\xi_1}
\newcommand{\xitw}{\xi_2}
\newcommand{\xith}{\xi_3}

\newcommand{\xo}{x_1}
\newcommand{\xtw}{x_2}
\newcommand{\xth}{x_3}
\newcommand{\xn}{x_n}
\newcommand{\xij}{\xi_j}

\newcommand{\etaj}{\eta_j}
\newcommand{\xil}{\xi_\ell}
\newcommand{\ubfk}{\ubf_k}
\newcommand{\ubfko}{\ubf_{k,1}}
\newcommand{\ubfkt}{\ubf_{k,2}}
\newcommand{\xk}{x_k}
\newcommand{\yk}{y_k}

\newcommand{\partialxjxl}{\partial^2_{x_j,x_\ell}}
\newcommand{\partialxjoxjn}{\partial^d_{x^{j_1}_1,\dots, x_n^{j_n}}}

\newcommand{\Ttilde}{\widetilde T}
\newcommand{\Ttildeo}{\Ttilde_\alpha}
\newcommand{\Ttildeoo}{\Ttilde_{\alpha,1}}
\newcommand{\Ttildeot}{\Ttilde_{\alpha,2}}
\newcommand{\Ttildeoeps}{\Ttilde_{\alpha,\varepsilon}}
\newcommand{\Ttildet}{\Ttilde_\beta}
\newcommand{\Ttildekjlo}{\Ttilde_{k,j\ell,1}}
\newcommand{\Ttildekjlt}{\Ttilde_{k,j\ell,2}}
\newcommand{\varphihat}{\widehat \varphi}
\newcommand{\Cphirhoz}{C_{\varphi,\rho,0}}
\newcommand{\Cphirhoo}{C_{\varphi,\rho,1}}
\newcommand{\Cphirhokbf}{C_{\varphi,\rho,\kbf}}
\newcommand{\Boz}{B_1(0)}
\newcommand{\Omegaaeps}{\Omega_{a,\varepsilon}}
\newcommand{\jbf}{\mathbf j}
\newcommand{\nbfOmega}{\nbf_\Omega}
\newcommand{\rbf}{\mathbf r}
\newcommand{\alphabold}{\boldsymbol \alpha}
\newcommand{\omegabold}{\boldsymbol \omega}
\newcommand{\ellbf}{\mathbf\ell}
\newcommand{\kbf}{\mathbf k}
\newcommand{\jo}{j_1}
\newcommand{\jtw}{j_2}
\newcommand{\jth}{j_3}
\newcommand{\CP}{C_P}
\newcommand{\Gbf}{\mathbf G}
\newcommand{\Gbftilde}{\widetilde{\Gbf}}
\newcommand{\ftilde}{\widetilde f}

\title{\Large A Ne\v cas-Lions inequality with symmetric gradients on star-shaped domains
based on a first order Babu\v ska-Aziz inequality}
\author{\large{Michele Botti\thanks{MOX Department of Mathematics, Politecnico of Milano, 20133 Milano, Italy (michele.botti@polimi.it)},   
Lorenzo Mascotto\thanks{Department of Mathematics and Applications, University of Milano-Bicocca, 20125 Milan, Italy (lorenzo.mascotto@unimib.it); 
Faculty of Mathematics, University of Vienna, 1090 Vienna, Austria;
IMATI-CNR, 27100 Pavia, Italy}}}
\date{}

\begin{document}
\maketitle

\begin{abstract}
\noindent
We prove a Ne\v cas-Lions inequality with symmetric gradients
on two and three dimensional domains of diameter~$R$
that are star-shaped with respect
to a ball of radius~$\rho$;
the constants in the inequality are explicit
with respect to~$R$ and~$\rho$.
Crucial tools in deriving this inequality
are a first order Babu\v ska-Aziz inequality
based on Bogovskiĭ's construction
of a right-inverse of the divergence
and Fourier transform techniques proposed by Dur\'an.
As a byproduct, we derive arbitrary order estimates
in arbitrary dimension for that operator.


\end{abstract}

\section{Introduction} \label{section:introduction}

We derive a Ne\v cas-Lions inequality with symmetric gradients
on star-shaped domains in two and three dimensions;
first order Babu\v ska-Aziz and (vector) Ne\v cas-Lions inequalities are crucial tools
to show estimates that are explicit
with respect to certain geometric quantities of the domain.
As a byproduct, we also derive explicit arbitrary order
Babu\v ska-Aziz inequalities in arbitrary dimension.

\paragraph*{Outline of the introduction.}
After introducing the functional setting and the domains of interest,
we review the literature and the main concepts related to
the lowest order Babu\v ska-Aziz and Ne\v cas-Lions inequalities
in Sections~\ref{subsection:standard-BA-ineq} and~\ref{subsection:standard-NL}.
In Sections~\ref{subsection:main-result1} and~\ref{subsection:main-result2},
we discuss the generalisation of these two results
to the first order case.
The Ne\v cas-Lions inequality with symmetric gradients,
based on all the foregoing results,
is shown in Section~\ref{subsection:main-result3}.
Finally, we describe the outline of the remainder of the paper.

\paragraph*{Functional spaces and notation.}
In what follows, $\nabla$, $\nabla\times$, and~$\div$ denote the
gradient, curl, and divergence operators.
The operators~$\nablaS$ and~$\nablaSS$ are
the symmetric and skew-symmetric parts of~$\nabla$:
\begin{equation} \label{splitting-gradient}
\nabla = \nablaS + \nablaSS.
\end{equation}
We use standard notation~\cite{Ern-Guermond:2021}
for Sobolev spaces on Lipschitz domains~$\Omega$
with boundary~$\partial \Omega$.
The outward unit normal vector to~$\partial \Omega$ is~$\nbfOmega$.
$H^s(\Omega)$ denotes the Sobolev space of order~$s\ge0$,
which we equip with inner product $(\cdot,\cdot)_s$,
seminorm~$\SemiNorm{\cdot}_{s,\Omega}$,
and norm~$\Norm{\cdot}_{s,\Omega}$.
The case~$s=0$ corresponds to~$H^0(\Omega)=L^2(\Omega)$.
The space of functions in $L^2(\Omega)$ with zero average
over~$\Omega$ is denoted by~$L^2_0(\Omega)$.

For~$s$ positive, we define $H^s_0(\Omega)$ as
the closure of $\Ccalzinf(\Omega)$ with respect to the $H^s(\Omega)$ norm.
In what follows, we shall particularly use the spaces~$H_0^1(\Omega)$
and~$H^2_0(\Omega)$,
which coincide~\cite{Meyers-Serrin:1964, Ciarlet:2002}
with the spaces of functions with zero trace,
and functions with zero trace and whose gradients have
zero trace over~$\partial\Omega$, respectively.

Negative order Sobolev spaces are defined by duality.
We introduce the spaces~$H^{-1}(\Omega) := [H^1_0(\Omega)]^*$
and~$H^{-2}(\Omega) := [H^2_0(\Omega)]^*$
equipped with the norms
\begin{equation} \label{negative-norms}
\Norm{u}_{-1,\Omega} 
: = \sup_{v\in H^1_0(\Omega)} \frac{_{-1}\langle u,v \rangle_1}{\SemiNorm{v}_{1,\Omega}},
\qquad\qquad\qquad
\Norm{u}_{-2,\Omega} 
: = \sup_{v\in H^2_0(\Omega)} \frac{_{-2}\langle u,v \rangle_2}{\SemiNorm{v}_{2,\Omega}},
\end{equation}
where~$_{-\ell}\langle \cdot,\cdot \rangle_\ell$
is the duality pairing between $H^{-\ell}(\Omega)$
and~$H^\ell_0(\Omega)$.

The definitions above extend to the case of
vector fields and tensors.
With an abuse of notation, the norms on scalar, vector fields,
and tensors are denoted with the same symbols.

For positive~$a$ and~$b$,
by~$a \lesssim b$, we shall occasionally mean
that there exists a positive constant~$c$
independent of relevant geometric parameters
such that~$a \le c \ b$.
An extra subscript makes it explicit
a hidden dependence on a parameter of interest.

\paragraph*{Domains of interest.}
Henceforth, $\Omega$ in~$\Rbb^n$
is a
\begin{equation}\label{Rrho}
\text{domain of diameter~$\hOmega$
that is star-shaped with respect to a ball $\Brho$ of radius~$\rho$}.
\end{equation}

\subsection{The lowest order Babu\v ska-Aziz inequality} \label{subsection:standard-BA-ineq}
The standard, lowest order version of the Babu\v ska-Aziz inequality
was proven as early as 1961 by Cattabriga~\cite{Cattabriga:1961}.
However, the name is associated to the authors
of the later work~\cite{Babuska-Aziz:1972}
and was assigned by Horgan and Payne~\cite{Horgan-Payne:1983};
see also~\cite{Costabel-Dauge:2015}.

The inequality reads as follows:
there exists a positive constant~$\CBAz$ such that
for any~$f$ in $L^2_0(\Omega)$
one can construct~$\ubf$ in $[H^1_0(\Omega)]^n$ satisfying
\begin{equation} \label{0th-Babu-Aziz}
\div \ubf = f,
\qquad\qquad\qquad
\SemiNorm{\ubf}_{1,\Omega}
\le \CBAz \Norm{f}_{0,\Omega}.
\end{equation}
The subscript in~$\CBAz$ relates to Babu\v ska-Aziz.
Explicit constructions of a vector field~$\ubf$
as in~\eqref{0th-Babu-Aziz}
may be performed in different ways;
here, we shall follow the approach by
Bogovskiĭ in~\cite{Bogovskii:1979, Bogovskii:1980},
where he showed a particular construction of the right inverse
of the divergence based on integral kernels.
Alternative avenues, which give less information on the constant~$\CBAz$
and construct less smooth right-inverses of the divergence,
are based on solving curl-div, diffusion, or Stokes problems;
see, e.g., \cite{Arnold-Scott-Vogelius:1988} and \cite[Lemma 11.2.3]{Brenner-Scott:2008};
as such, higher order estimates based on this approach
require extra regularity assumptions on the boundary of the domain,
which are instead not needed following Bogovskiĭ's approach.

\paragraph*{Minimal literature on the lowest order Babu\v ska-Aziz inequality.}
The literature associated with inequality~\eqref{0th-Babu-Aziz} is widespread.
We refer to \cite[pp. 227-228]{Galdi:2011}, \cite{Guzman-Salgado:2021},
and~\cite{Costabel-Dauge:2015} for a thorough historical review.

Here, we only mention that the divergence problem in~\eqref{0th-Babu-Aziz}
was raised as early as in 1961 by Cattabriga~\cite{Cattabriga:1961};
see also the later works~\cite{Ladyzhenskaya:1967, Necas:1967}.
Bogovskiĭ introduced an explicit representation for~$\vbf$ solving~\eqref{0th-Babu-Aziz}
in~\cite{Bogovskii:1979, Bogovskii:1980}.
Several references discuss the validity of similar estimates;
for instance, implicit constants for general Lipschitz domains are available
in \cite[Theorems~2.4 and~2.9]{Borchers-Sohr:1990}.
The case of negative Sobolev norms is described in~\cite{Geissert-Heck-Hieber:2006}.

The explicit dependence of the constant~$\CBAz$ on geometric parameters of~$\Omega$
is studied in fewer references;
see~\cite{Acosta-Duran:2017} for a list.
For~$R$ and~$\rho$ as in~\eqref{Rrho},
Galdi~\cite{Galdi:2011} gives estimates of the form
\begin{equation} \label{suboptimal-Galdi}
\CBAz \lesssim_n \left(\frac{\hOmega}{\rho}\right)^{n+1}.
\end{equation}
The main tool in the analysis is the one
discussed originally by Bogovskiĭ~\cite{Bogovskii:1979, Bogovskii:1980},
i.e., the Calder\'on-Zygmund singular integral operator theory.

Improved estimates of the form
\begin{equation} \label{explicit-estimates-Duran}
\CBAz \lesssim_n \frac{R}{\rho} 
    \left( \frac{\SemiNorm{\Omega}}{\vert \Brho \vert} \right) ^{\frac{n-2}{2(n-1)}}
    \left( \log \frac{\SemiNorm{\Omega}}{\vert \Brho \vert} \right)^{\frac{n}{2(n-1)}}
\end{equation}
were proven by Dur\'an~\cite{Duran:2012} based on the properties
of the Fourier transform.
In the same reference, it is shown that the estimates
are optimal up to the logarithmic factor for $n=2$.
More precisely, a 2D counterexample is exhibited showing that
the following holds true:
\[
\CBAz \gtrsim \frac{\hOmega}{\rho}.
\]
For the two dimensional case,
Costabel and Dauge~\cite[Theorem~2.3]{Costabel-Dauge:2015} proved
that the logarithmic factor in Dur\'an's estimates can be removed
again for $n=2$.

\paragraph*{Minimal literature on higher order Babu\v ska-Aziz inequalities.}
Higher order Babu\v ska-Aziz inequalities are far less investigated.
They are stated in the original paper by Bogovskiĭ~\cite{Bogovskii:1979}
without mention on the behaviour of the constants.
Galdi~\cite[Remark III.3.2]{Galdi:2011}
claims that similar bounds to the lowest order case can be derived;
however, no explicit constants are given in that case sa well;
the analysis hinges upon the Calder\'on-Zygmund theory.
Costabel and McIntosh prove arbitrary order estimates~\cite{Costabel-McIntosh:2010}
without explicit dependence on the geometry on the domain.
Guzm\'an and Salgado~\cite{Guzman-Salgado:2021}
prove an explicit first order generalised Poincar\'e inequality,
which is related to Bogovskiĭ's operator
without imposition of boundary conditions,
and give a road map on how to prove higher order explicit estimates;
tools as those in~\cite{Duran:2012} are employed;
no estimates are given for Bogovskiĭ's operator.

\subsection{The lowest order Ne\v cas-Lions inequality}
\label{subsection:standard-NL}
The standard, lowest order Ne\v cas-Lions inequality
is a very well known result in the theory of Sobolev spaces.
It is proven in the book by Ne\v cas \cite[Lemma 3.7.1]{Necas:1967};
the connection to the name of Lions is less clear,
and is probably due to~\cite[Note 27, page 320]{Magenes-Stampacchia:1958},
where the result is mentioned as a private communication
by Lions himself to Magenes and Stampacchia, yet without an explicit proof.

Given~$\Piz: L^1(\Omega) \to \Rbb$ the average operator over~$\Omega$,
the inequality reads as follows:
there exists a positive constant~$\CNLz$ such that
\begin{equation} \label{0th-Necas-Lions}
\Norm{f -\Piz f}_{0,\Omega}
\le \CNLz \Norm{\nabla f}_{-1,\Omega} 
 \qquad\qquad \forall f \in L^2(\Omega).
\end{equation}
The subscript in~$\CNLz$ relates to Ne\v cas-Lions.
An equivalent statement for~\eqref{0th-Necas-Lions}
is that the following constants are bounded from above
and below, respectively:
\begin{equation} \label{standard:Necas-Lions-pre}
\CNLz:= \sup_{f \in L^2(\Omega)}
        \frac{\Norm{f -\Piz f}_{0,\Omega}}{\Norm{\nabla f}_{-1,\Omega}},
\qquad\qquad
\CNLz^{-1}:= \inf_{f \in L^2(\Omega)}
        \frac{\Norm{\nabla f}_{-1,\Omega}}{\Norm{f -\Piz f}_{0,\Omega}}.
\end{equation}
The constants~$\CBAz$ and~$\CNLz$
in~\eqref{0th-Babu-Aziz} and~\eqref{0th-Necas-Lions} are related
to constants in other relevant inequalities in Sobolev spaces
as well, including the standard inf-sup constant~$\betaz$
defined as
\begin{equation} \label{0th-inf-sup}
\inf_{f\in L^2_0(\Omega)}\sup_{\ubf \in [H^1_0(\Omega)]^n}
        \frac{(\div\ubf,f)_{0,\Omega}}{\Norm{\ubf}_{1,\Omega} \Norm{f}_{0,\Omega}}
        =: \betaz.
\end{equation}
\begin{proposition} \label{proposition:relation-constants-0th}
Let~$\CBAz$, $\CNLz$, and~$\betaz$ be given
in~\eqref{0th-Babu-Aziz}, \eqref{0th-Necas-Lions},
and~\eqref{0th-inf-sup}.
Then, the following holds true:
\[
\CBAz \ge \CNLz = \betaz^{-1}.
\]
\end{proposition}
\begin{proof}
For all~$f$ in~$L^2_0(\Omega)$, the definition
of negative norms in~\eqref{negative-norms} implies
\[
\sup_{\ubf \in [H^1_0(\Omega)]^n}
        \frac{(\div\ubf,f)_{0,\Omega}}{\Norm{\ubf}_{1,\Omega} \Norm{f}_{0,\Omega}}
= \sup_{\ubf \in [H^1_0(\Omega)]^n}
        \frac{_{-1}\langle \nabla f, \ubf \rangle_1}{\Norm{\ubf}_{1,\Omega} \Norm{f}_{0,\Omega}}
= \frac{\Norm{\nabla f}_{-1,\Omega}}{\Norm{f}_{0,\Omega}} . 
\]
We take the $\inf$ over all possible~$f$ in~$L^2_0(\Omega)$,
use~\eqref{standard:Necas-Lions-pre},
and deduce that~$\betaz = \CNLz^{-1}$.

On the other hand, for all~$f$ in~$L^2_0(\Omega)$,
we can consider a specific~$\ubf$ satisfying~\eqref{0th-Babu-Aziz},
which gives
\[
\sup_{\ubf \in [H^1_0(\Omega)]^n}
        \frac{(\div\ubf,f)_{0,\Omega}}{\Norm{\ubf}_{1,\Omega} \Norm{f}_{0,\Omega}}
\ge \frac{\Norm{f}_{0,\Omega}^2}{\Norm{\ubf}_{1,\Omega} \Norm{f}_{0,\Omega}}
\ge \CBAz^{-1}. 
\]
Taking the inf over all possible $f$ in $L^2_0(\Omega)$
and recalling the standard inf-sup condition~\eqref{0th-inf-sup}
give $\betaz \ge \CBAz^{-1}$.
The assertion follows.
\end{proof}

Since an upper bound on~$\CBAz$ is available from Lemma~\ref{0th-Babu-Aziz},
which is explicit in terms of~$n$, $R$, and~$\rho$ as in~\eqref{Rrho},
then Proposition~\ref{proposition:relation-constants-0th}
implies an upper bound for~$\CNLz$
and a lower bound for~$\betaz$
with the same explicit dependence.
The relation with constants appearing in other inequalities
is discussed, amongst others,
in~\cite{Amrouche-Ciarlet-Mardare:2015, Ciarlet:2002, Costabel-Dauge:2015, Horgan-Payne:1983}.

\subsection{Main result~$1$: a first order Babu\v ska-Aziz inequality} \label{subsection:main-result1}
An important tool in the proof of Theorem~\ref{theorem:new-Necas-Lions-symmetric-gradients} below
is the proof of a Babu\v ska-Aziz inequality,
based on first order estimates for Bogovskiĭ's construction
of the right-inverse of the divergence.
More precisely,
there exist positive constants~$\CBAoA$ and~$\CBAoB$
such that for all~$f$ in~$H^1_0(\Omega) \cap L^2_0(\Omega)$,
one can construct~$\ubf$ in~$[H^2_0(\Omega)]^n$ satisfying
\begin{equation} \label{1st-Babu-Aziz}
    \div\ubf = f,
    \qquad\qquad\qquad
    \SemiNorm{\ubf}_{2,\Omega}
    \le \CBAoA \Norm{f}_{0,\Omega} 
        + \CBAoB \SemiNorm{f}_{1,\Omega}.
\end{equation}
We state the result here and postpone its proof
to Section~\ref{section:2nd-Bogovskii} below.

\begin{theorem}[A first order Babu\v ska-Aziz inequality] \label{theorem:1st-Babu-Aziz}
Let~$\ubf$ and~$f$ be as in~\eqref{1st-Babu-Aziz},
and~$\Omega$, $\Brho$, $R$, and~$\rho$ be as in~\eqref{Rrho}.
Then, inequality~\eqref{1st-Babu-Aziz} holds true with
\begin{equation} \label{estimates:Bogo-constants}
\CBAoA \lesssim \frac{R}{\rho^2}
            \left[ 1 + \left( \frac{\SemiNorm{\Omega}}{\SemiNorm{\Brho}} \right)^{\frac{n-2}{2(n-1)}}
            \left( \log\frac{\SemiNorm{\Omega}}{\SemiNorm{\Brho}} \right)^{\frac{n}{2(n-1)}}
            \right],
\qquad\qquad\qquad
\CBAoB \lesssim \frac{R}{\rho}.
\end{equation}
\end{theorem}

We provide the reader with some comments on the optimality
of the constant~$\CBAoA$ in~\eqref{estimates:Bogo-constants}
in Section~\ref{subsection:counterexample} below.

\subsection{Main result~$2$: a first order Ne\v cas-Lions inequality} \label{subsection:main-result2}
Introduce the space $H^{-1}(\Omega) / \Rbb$,
which is the space~$H^{-1}(\Omega)$ equipped with the norm
\[
\Norm{f}_{H^{-1}(\Omega) / \Rbb}
:= \inf_{c \in \Rbb} \Norm{f-c}_{-1,\Omega}.
\]
Recall the negative norms in~\eqref{negative-norms}.
We discuss a first order Ne\v cas-Lions inequality (for vectors):
there exists a positive constant~$\CNLo$ such that
\begin{equation} \label{1st-Necas-Lions}
\Norm{f}_{H^{-1}(\Omega) / \Rbb}
\le \CNLo \Norm{\nabla f}_{-2,\Omega}
\qquad\qquad \forall f \in H^{-1}(\Omega)/ \Rbb.
\end{equation}
An equivalent statement for~\eqref{1st-Necas-Lions}
is that the following constants are bounded from above
and below, respectively:
\begin{equation} \label{CNLo}
\CNLo= \sup_{f \in H^{-1}(\Omega)/ \Rbb} 
           \frac{\Norm{f}_{H^{-1}(\Omega) / \Rbb}}
                {\Norm{\nabla f}_{-2,\Omega}},
\qquad\qquad
\CNLo^{-1} = \inf_{f \in H^{-1}(\Omega)/ \Rbb} 
           \frac{\Norm{\nabla f}_{-2,\Omega}} 
            {\Norm{f}_{H^{-1}(\Omega) / \Rbb}}.
\end{equation}
The constants~$\CBAoA$ and~$\CBAoB$, and~$\CNLo$
in~\eqref{1st-Babu-Aziz} and~\eqref{estimates:Bogo-constants},
and~\eqref{1st-Necas-Lions} are related
to constants in other relevant inequalities in Sobolev spaces
as well, including the first order inf-sup constant~$\betao$
defined as
\begin{equation} \label{1st-inf-sup}
\inf_{f \in H^{-1}(\Omega) / \Rbb}
\sup_{\ubf \in [H^2_0(\Omega)]^n}
\frac{_{-1}\langle f, \divbf \ubf \rangle_{1}}{\Norm{f}_{H^{-1}(\Omega) / \Rbb} \SemiNorm{\ubf}_{2,\Omega}}
=: \betao
\end{equation}
and the positive constant~$\CP$ appearing in the Poincar\'e inequality
\begin{equation} \label{Poincare}
\Norm{f}_{0,\Omega}
\le \CP R \SemiNorm{f}_{1,\Omega}
\qquad\qquad \forall f \in H^1_0(\Omega).
\end{equation}
The constant~$\CP$ is independent of~$R$ and~$\rho$ in~\eqref{Rrho};
see, e.g., \cite[Section~3.3]{Ern-Guermond:2021}.

The following result is the first order version of Proposition~\ref{proposition:relation-constants-0th}.
\begin{proposition} \label{proposition:1st-Necas-Lions}
Let~$\CBAoA$ and~$\CBAoB$, $\CNLo$, $\betao$, and~$\CP$ be given
in~\eqref{1st-Babu-Aziz} and~\eqref{estimates:Bogo-constants},
\eqref{1st-Necas-Lions}, \eqref{1st-inf-sup}, and~\eqref{Poincare}.
Then, the following holds true:
\begin{equation} \label{CBo-le-CNLo}
\CBAoA \CP R+\CBAoB \ge \CNLo = \betao^{-1}.
\end{equation}
\end{proposition}
\begin{proof}
For all~$f$ in $H^{-1}(\Omega) / \Rbb$,
an integration by parts and the definition
of negative Sobolev norms in~\eqref{negative-norms} imply
\[
\sup_{\ubf \in [H^2_0(\Omega)]^{n}}
\frac{_{-1}\langle f, \div \ubf \rangle_{1}}{\Norm{f}_{H^{-1}(\Omega) / \Rbb}
        \SemiNorm{\ubf}_{2,\Omega}}
= \sup_{\ubf \in [H^2_0(\Omega)]^{n}}
\frac{_{-2}\langle \nabla f, \ubf \rangle_{2}}{\Norm{f}_{H^{-1}(\Omega) / \Rbb} \SemiNorm{\ubf}_{2,\Omega}}
=: \frac{\Norm{\nabla f}_{-2,\Omega}}{\Norm{f}_{H^{-1}(\Omega) / \Rbb}}.
\]
We take the infimum over all such possible~$f$
and exploit the identities
\[
\betao
\overset{\eqref{1st-inf-sup}}{=}
\inf_{f \in H^{-1}(\Omega) / \Rbb}
\frac{\Norm{\nabla f}_{-2,\Omega}}{\Norm{f}_{H^{-1}(\Omega) / \Rbb}}
= \left( \sup_{f \in H^{-1}(\Omega) / \Rbb}
    \frac{\Norm{\nabla f}_{-2,\Omega}}{\Norm{f}_{H^{-1}(\Omega) / \Rbb}} \right)^{-1}
\overset{\eqref{CNLo}}{=} \CNLo^{-1},
\]
which implies $\betao=\CNLo^{-1}$.

On the other hand, \eqref{1st-Babu-Aziz} guarantees for all~$\ftilde$
in~$H^1_0(\Omega)$ the existence
of~$\ubf$ in $[H^2_0(\Omega)]^{n}$ such that
\[
\div \ubf = \ftilde,
\qquad\qquad
\SemiNorm{\ubf}_{2,\Omega}
\le \CBAoA \Norm{\ftilde}_{0,\Omega}
   + \CBAoB \SemiNorm{\ftilde}_{1,\Omega}.
\]
This and the Poincar\'e inequality~\eqref{Poincare} give
\[
\begin{split}
& \sup_{\ubf \in [H^2_0(\Omega)]^{n}}
\frac{_{-1}\langle f, \div \ubf \rangle_{1}}{\Norm{f}_{H^{-1}(\Omega) / \Rbb} \SemiNorm{\ubf}_{2,\Omega}}
 \ge \sup_{\ftilde \in H^1_0(\Omega)}
\frac{_{-1}\langle f, \ftilde \rangle_{1}}
    {\Norm{f}_{H^{-1}(\Omega) / \Rbb}
    [\CBAoA \Norm{\ftilde}_{0,\Omega} + \CBAoB \SemiNorm{\ftilde}_{1,\Omega}]} \\
& \ge \sup_{\ftilde \in H^1_0(\Omega)}
\frac{_{-1}\langle f, \ftilde \rangle_{1}}
    {(\CBAoA\CP R+\CBAoB)\Norm{f}_{H^{-1}(\Omega) / \Rbb}
    \SemiNorm{\ftilde}_{1,\Omega}}
\overset{\eqref{negative-norms}}{=}
(\CBAoA\CP R+\CBAoB)^{-1}.
\end{split}
\]
We take the infimum over all $f$ in $H^{-1}(\Omega) / \Rbb$,
recall the first order inf-sup condition~\eqref{1st-inf-sup},
and deduce $\betao \ge (\CBAoA\CP R+\CBAoB)^{-1}$.
The assertion follows.
\end{proof}

Since an upper bound on~$\CBAoA$ and~$\CBAoB$ is available from Theorem~\ref{theorem:1st-Babu-Aziz},
which is explicit in terms of~$n$, $R$, and~$\rho$ as in~\eqref{Rrho},
then Proposition~\ref{proposition:1st-Necas-Lions}
implies an upper bound for~$\CNLo$
and a lower bound for~$\betao$
with the same explicit dependence.
For more general Ne\v cas-Lions inequalities,
yet with unknown constants,
see~\cite{Amrouche-Girault:1994} and the references therein.

\subsection{Main result~$3$: a Ne\v cas-Lions inequality with symmetric gradients} \label{subsection:main-result3}
The spaces~$\RM(\Omega)$ of rigid body motions
in two and three dimensions have
cardinality~$3$ and~$6$, and are given by
\[
\RM(\Omega)
:= \begin{cases}
\left\{ \rbf(x) = \alphabold + b \mathbf (x_2,-x_1)^{\mathrm T}
\text{ for any }
\alphabold \in \Rbb^2,\; b \in \Rbb \right\} & \text{ in $2$D}\\
\left\{ \rbf(x) = \alphabold +\omegabold \times (x_1,x_2,x_3)^{\mathrm T}
\text{ for any }
\alphabold, \omegabold \in \Rbb^3 \right\}  & \text{ in $3$D}.
\end{cases}
\]
Let~$\PiRM$ denote the $L^2(\Omega)$ projection onto~$\RM(\Omega)$.
We further introduce the space of symmetric tensors
\[
\Sigmabold
:=\{  \taubold \in H(\div,\Omega)  \mid \taubold \text{ is symmetric} \},
\]
which we endow with the norm
\begin{equation} \label{Sigmabold-norm}
\Norm{\taubold}_{\Sigmabold}^2
:=  \Norm{\taubold}_{0,\Omega}^2
    + R^2 \Norm{\div\taubold}_{0,\Omega}^2 .
\end{equation}
Note that
\begin{equation} \label{symmetric-tensors-property}
_{-1}\langle \nabla \vbf,\taubold \rangle_1
= _{-1}\langle \nablaS \vbf,\taubold \rangle_1 
\qquad\qquad
\forall \vbf \in [L^2(\Omega)]^2,\;
    \taubold\in\Sigmabold.
\end{equation}
We state a Ne\v cas-Lions inequality with symmetric gradients
on two and three dimensional domains,
which is explicit in terms of~$R$, $\rho$, and~$n$ as in~\eqref{Rrho}.

We state the inequality here
and postpone its proof to Section~\ref{section:NL-sg} below.    

\begin{theorem}[A Ne\v cas-Lions inequality with symmetric gradients] \label{theorem:new-Necas-Lions-symmetric-gradients}
There exists a positive constant~$\CNLzstar$
depending only on~$n$, $R$, and~$\rho$ as in~\eqref{Rrho}
through~$\CNLz$ in~\eqref{0th-Necas-Lions},
$\CBAoA$ and~$\CBAoB$ in~\eqref{1st-Babu-Aziz},
and~$\CP$ in~\eqref{Poincare}, such that
\begin{equation} \label{Necas-Lions-symmetric-tensors}
\begin{split}
\Norm{\vbf -\PiRM \vbf}_{0,\Omega}
& \le \CNLzstar \Norm{\nablaS \vbf}_{-1,\Omega}  \\
& \le \CNLz \left[ 1+\sqrt2 \left(\CBAoA\CP R+\CBAoB\right) \right] 
    \Norm{\nablaS \vbf}_{-1,\Omega} 
\qquad \forall \vbf \in [L^2(\Omega)]^n.
\end{split}
\end{equation}
\end{theorem}
Since upper bounds on~$\CBAoA$ and~$\CBAoB$,
and $\CNLz$ are available
from Lemma~\ref{1st-Babu-Aziz},
and Proposition~\ref{proposition:relation-constants-0th}
and display~\eqref{explicit-estimates-Duran},
which are explicit in terms of~$n$, $R$, and~$\rho$ as in~\eqref{Rrho},
(bounds on~$\CP$ are standard)
then Theorem~\ref{theorem:new-Necas-Lions-symmetric-gradients}
implies an upper bound for~$\CNLzstar$
and a lower bound for~$\betao$
with the same explicit dependence.
Roughly speaking,
Theorem~\ref{theorem:new-Necas-Lions-symmetric-gradients}
is a Korn-type version
of the standard lowest order Ne\v cas-Lions inequality.

Introduce the spaces
\[
\widetilde{\Sigmabold}
:= \{ \taubold \in \Sigmabold \mid \langle\taubold\ \nbf,
        \boldsymbol{w}\rangle_{\partial\Omega}=0 
        \quad\forall\boldsymbol{w}\in [H^1(\Omega)]^n\};
\qquad\qquad
\widetilde{\Vbf}
:=  \{\vbf\in[L^2(\Omega)]^n\mid \PiRM\vbf = 0\}.
\]
A consequence of Theorem~\ref{theorem:new-Necas-Lions-symmetric-gradients}
is an inf-sup condition,
which is of great importance in the analysis of the mixed
(Hellinger-Reissner) formulation
of linear elasticity problems:
there exists a positive constant~$\betastarz$
such that
\begin{equation} \label{inf-sup:symmetric-tensors}
\begin{split}
\inf_{\vbf\in\widetilde{\Vbf}}\ \sup_{\taubold\in\widetilde{\Sigmabold}} \frac{(\divbf \taubold, \vbf)_{0,\Omega}}{\Norm{\taubold}_{\Sigmabold} \Norm{\vbf}_{0,\Omega}}
=: \betastarz R^{-1}.
\end{split}
\end{equation}
\begin{proposition} \label{proposition:inf-sup:symmetric-tensors}
Let~$\CNLzstar$ and~$\betastarz$
be given in~\eqref{Necas-Lions-symmetric-tensors}
and~\eqref{inf-sup:symmetric-tensors}.
Then, the inf-sup condition~\eqref{inf-sup:symmetric-tensors} holds true with
\[
\betastarz \ge (\CNLzstar)^{-1}(1+\CP^2)^{-\frac12}.
\]
\end{proposition}
\begin{proof}
For all tensors~$\taubold$, let~$\taubold_S$
denote its symmetric part.
We have
\[
\begin{split}
\betastarz
& = \inf_{\vbf\in\widetilde{\Vbf}}\ \sup_{\taubold\in\widetilde{\Sigmabold}} 
        \frac{(\divbf \taubold, \vbf)_{0,\Omega}}{\Norm{\taubold}_{\Sigmabold} \Norm{\vbf}_{0,\Omega}}
  \ge \inf_{\vbf\in\widetilde{\Vbf}}\ 
    \sup_{\taubold\in\Sigmaboldtilde \cap [H^1_0(\Omega)]^{n\times n}} 
        \frac{(\divbf \taubold, \vbf)_{0,\Omega}}{\Norm{\taubold}_{\Sigmabold} \Norm{\vbf}_{0,\Omega}}\\
& \overset{\text{IBP}}{=}
    \inf_{\vbf\in\widetilde{\Vbf}}\
    \sup_{\taubold\in\Sigmaboldtilde \cap [H^1_0(\Omega)]^{n\times n}} 
        \frac{_{-1}\langle \nabla \vbf, \taubold \rangle_1}{\Norm{\taubold}_{\Sigmabold} \Norm{\vbf}_{0,\Omega}}
\overset{\eqref{symmetric-tensors-property}}{=}
    \inf_{\vbf\in\widetilde{\Vbf}}\ 
    \sup_{\taubold\in\Sigmaboldtilde \cap [H^1_0(\Omega)]^{n\times n}} 
        \frac{_{-1}\langle \nablaS \vbf, \taubold \rangle_1}{\Norm{\taubold}_{\Sigmabold} \Norm{\vbf}_{0,\Omega}}\\
& \overset{\eqref{Sigmabold-norm}}{\ge}
    \inf_{\vbf\in\widetilde{\Vbf}}\ 
    \sup_{\taubold\in\Sigmaboldtilde \cap [H^1_0(\Omega)]^{n\times n}} 
        \frac{_{-1}\langle \nablaS \vbf, \taubold \rangle_1}
        {({\Norm{\taubold}_{0,\Omega}^2 + R^{2}\SemiNorm{\taubold}_{1,\Omega}^2)^\frac12} \Norm{\vbf}_{0,\Omega}}\\
& \overset{\eqref{Poincare}}\ge 
     \inf_{\vbf\in\widetilde{\Vbf}}\ 
    \sup_{\taubold\in\Sigmaboldtilde \cap [H^1_0(\Omega)]^{n\times n}} 
        \frac{_{-1}\langle \nablaS \vbf,  \taubold \rangle_1}{(1+\CP^2)^\frac12 R \SemiNorm{\taubold}_{1,\Omega}
        \Norm{\vbf}_{0,\Omega}}\\
& = \inf_{\vbf\in\widetilde{\Vbf}}\ 
    \sup_{\taubold\in [H^1_0(\Omega)]^{n\times n}} 
        \frac{_{-1}\langle \nablaS \vbf,  \taubold_S \rangle_1}{(1+\CP^2)^\frac12 R \SemiNorm{\taubold_S}_{1,\Omega}
        \Norm{\vbf}_{0,\Omega}}.
\end{split}
\]
Since $\SemiNorm{\taubold_S}_{1,\Omega}
\le \SemiNorm{\taubold}_{1,\Omega}$
for all tensors~$\taubold$
and the numerator involves~$\nablaS\vbf$,
we deduce the assertion:
\[
\begin{split}
\betastarz
& \ge \inf_{\vbf\in\widetilde{\Vbf}}\ 
    \sup_{\taubold\in [H^1_0(\Omega)]^{n\times n}} 
        \frac{_{-1}\langle \nablaS \vbf, \taubold \rangle_1}{(1+\CP^2)^\frac12 R \SemiNorm{\taubold}_{1,\Omega}
        \Norm{\vbf}_{0,\Omega}} \\
&  \overset{\eqref{negative-norms}}{=}
    (1+\CP^2)^{-\frac12}R^{-1}
  \inf_{\vbf\in\widetilde{\Vbf}}
        \frac{\Norm{\nablaS \vbf}_{-1,\Omega}}{\Norm{\vbf}_{0,\Omega}}
  \overset{\eqref{Necas-Lions-symmetric-tensors}}{\ge}
  (1+\CP^2)^{-\frac12}R^{-1} (\CNLzstar)^{-1} .
\end{split}
\]
\end{proof}

\paragraph*{Outline of the remainder of the paper.}
In Section~\ref{section:2nd-Bogovskii}, we prove Theorem~\ref{theorem:1st-Babu-Aziz},
whereas in Section~\ref{section:NL-sg}, we prove Theorem~\ref{theorem:new-Necas-Lions-symmetric-gradients}.
We also prove an arbitrary order version
of Theorem~\ref{theorem:1st-Babu-Aziz}
in Appendix~\ref{appendix:general-order-BA}.

\section{Proof of a first order Babu\v ska-Aziz inequality} \label{section:2nd-Bogovskii}
In this section, we prove Theorem~\ref{theorem:1st-Babu-Aziz} in several steps
and further discuss the optimality of the bounds on the constants therein.
To this aim, we follow Bogovskiĭ's construction~\cite{Bogovskii:1979}
of a right-inverse  of the divergence
and generalise Dur\'an's analysis~\cite{Duran:2012} to the first order case.

\paragraph*{Explicit construction of Bogovskiĭ's right-inverse of the divergence.}
\begin{subequations}\label{eq:bogovskii}
Consider~$\omega$ in~$\Ccalzinf(\Omega)$ with
\[
\int_\Omega \omega(x) \dx = 1,
\qquad\qquad\qquad
\supp(\omega) \subset \Brho .
\]
Given
\begin{equation}
\label{eq:kernel_def}
\Gbf:\Omega\times\Omega \to \Rbb^n,
 \qquad\qquad
 \Gbf(x,y)
 := \int_0^1 \frac{x-y}{t} \  \omega \left( y + \frac{x-y}{t}  \right) \frac{\dt}{t^n},
\end{equation}
we define
\begin{equation}\label{eq:Bogo_rightinv}
\ubf(x) := \int_\Omega  \Gbf(x,y) f(y) \dy.
\end{equation}
\end{subequations}

\subsection{Preliminary results} \label{subsection:preliminary-results}
We recall basic properties of the Fourier transform.
Given~$f$ in~$L^1(\Rbbn)$, we define its Fourier transform as
\begin{equation} \label{Fourier-transform}
\fhat(\xi) 
:= \int_{\Rbbn} e^{-2\pi {\rm i} x \cdot \xi} f(x) \dx.
\end{equation}
If~$f$ is in~$L^2(\Omega)$, we have the isometry
\begin{equation} \label{isometry:Fourier}
\Norm{f}_{0,\Rbbn} = \Norm{\fhat}_{0,\Rbbn}
\end{equation}
and the following property on the derivatives of the Fourier transform:
\begin{equation} \label{Fourier-derivative-property}
    \widehat{\partialxj f} (\xi)
    = 2\pi {\rm i} \xij \fhat(\xi)
    \qquad\qquad\forall j=1,\dots,n.
\end{equation}
We consider the following splitting of each component~$k$,
$k=1,\dots,n$, of~$\ubf$:
\[
\begin{split}
\ubfk
& :=\ubfko-\ubfkt \\
& := \int_0^1\int_{\Rbbn} \left(\yk + \frac{\xk-\yk}{t}\right)  
        \omega \left( y + \frac{x-y}{t}  \right) f(y) \dy \frac{\dt}{t^n}\\
& \quad - \int_0^1\int_{\Rbbn} \yk \
        \omega \left( y + \frac{x-y}{t}  \right) f(y) \dy \frac{\dt}{t^n} .
\end{split}
\]
In order to take derivatives of~$\ubfk$,
it is convenient to take a limit
in the sense of distributions \cite[Section~2]{Duran:2012}:
\[
\begin{split}
\ubfko
& = \lim_{\varepsilon\to 0}
 \int_\varepsilon^1\int_{\Rbbn} \left(\yk + \frac{\xk-\yk}{t}\right)  
        \omega \left( y + \frac{x-y}{t}  \right) f(y) \dy \frac{\dt}{t^n},\\
\ubfkt
& = \lim_{\varepsilon\to0} \int_\varepsilon^1\int_{\Rbbn} \yk \
        \omega \left( y + \frac{x-y}{t}  \right) f(y) \dy \frac{\dt}{t^n}.
\end{split}
\]
By doing that, we can interchange the derivative with the limit $\varepsilon \to 0$
and then pass with the derivative under the integral symbol;
in fact, the functions under the integral are in~$L^1$ and admit~$L^1$ derivatives.
\medskip

We take the second derivatives of~$\ubfko$ and~$\ubfkt$
with respect to the $j$-th and $\ell$-th directions
(without loss of generality we assume $j$ different from~$\ell$),
and get
\begin{equation} \label{splitting:Ttilde-2nd-derivative}
(\partialxjxl \ubfk)(x)
= [\Ttildekjlo (f(y)) - \Ttildekjlt (\yk f(y))](x).
\end{equation}
In~\eqref{splitting:Ttilde-2nd-derivative}, given~$\Ttilde$ any of the two operators
$\Ttildekjlo$ and~$\Ttildekjlt$, we let
\[
\Ttilde (g)(x)
:= \lim_{\varepsilon\to0}
  \int_{\varepsilon}^1 \int_{\Rbbn} \partialxjxl
        \left[ \varphi\left( y+ \frac{x-y}{t} \right) \right]
        g(y) \dy \frac{\dt}{t^n},
\]
where, for all $j,\ell=1,\dots,n$, we have set
\begin{equation} \label{options-varphi}
g(y) :=
\begin{cases}
f(y)        & \text{if } \Ttilde = \Ttildekjlo, \\
y_k f(y)    & \text{if } \Ttilde = \Ttildekjlt,
\end{cases}
\qquad\qquad
\varphi(x) :=
\begin{cases}
\xk \omega(x)   & \text{if } \Ttilde = \Ttildekjlo, \\
\omega(x)       & \text{if } \Ttilde = \Ttildekjlt.
\end{cases}
\end{equation}
In the forthcoming sections,
we shall prove the continuity of the operators
in~\eqref{splitting:Ttilde-2nd-derivative}.
To this aim, we henceforth fix~$j$ and~$\ell$, and consider the decomposition
\begin{equation} \label{def:Ttilde}
\Ttilde g
:= \Ttildeo g + \Ttildet g,
\end{equation}
where
\begin{equation} \label{def:Ttildeo}
\Ttildeo g(x)
:= \lim_{\varepsilon\to0}
  \int_{\varepsilon}^{\frac12} \int_{\Rbbn} \partialxjxl
        \left[ \varphi\left( y+ \frac{x-y}{t} \right) \right]
        g(y) \dy \frac{\dt}{t^n}
\end{equation}
and
\begin{equation} \label{def:Ttildet}
\Ttildet g(x)
:= \int_{\frac12}^1 \int_{\Rbbn} \partialxjxl
        \left[ \varphi\left( y+ \frac{x-y}{t} \right) \right]
        g(y) \dy \frac{\dt}{t^n}.
\end{equation}
The continuity estimates will follow
summing over all~$j$ and~$\ell$.

\subsection{Continuity of~$\Ttildeo$}
\label{subsection:continuity-Ttildeo}
We discuss the continuity of the operator in~\eqref{def:Ttildeo}.
We proceed in several steps.
First, we prove some properties of the Fourier transform
of~$\Ttildeo (g)$.

\begin{lemma} \label{lemma:1-Ttilde1}
If~$g$ in~\eqref{options-varphi} belongs to~$\Ccalzinf(\Rbbn)$,
then the following identity is valid:
\begin{equation} \label{splitting:Ttildeoo+Ttildeot}
\begin{split}
\widehat{\Ttildeo g}(\xi)
& = (2\pi {\rm i}  \xij) 
   \lim_{\varepsilon\to 0} 
   \left[ \int_{\varepsilon}^{\frac12}
    \varphihat\left( t \ \xi \right)
    \widehat{\partialxl g}(\xi) \dt
    + \int_{\varepsilon}^{\frac12}
    \widehat{\partialxl \varphi} (t \xi)
    \ghat((1-t)\xi) \dt \right]\\
& =: \widehat{\Ttildeoo g}(\xi)
    + \widehat{\Ttildeot g}(\xi).
\end{split}
\end{equation}
\end{lemma}
\begin{proof}
By definition, we write
\[
\Ttildeo g(x) =
\lim_{\varepsilon\to 0} \Ttildeoeps g(x).
\]
For all positive~$\varepsilon$, we write
\[
\widehat{\Ttildeoeps g}(\xi)
= \int_{\Rbbn} \int_{\varepsilon}^{\frac12} 
    \int_{\Rbbn} \partialxjxl
        \left[ \varphi\left( y+ \frac{x-y}{t} \right) \right]
        e^{-2 \pi {\rm i} x\cdot \xi} 
        g(y) \dy \frac{\dt}{t^n} \dx.
\]
Due to the regularity of~$g$ and~$\varphi$,
the integral exists.
Therefore, we can change the order of the integrals,
integrate by parts twice,
use the change of variable $z=y+(x-y)/t$,
the definition of the Fourier transform~\eqref{Fourier-transform} twice,
and~\eqref{Fourier-derivative-property},
recall that the support of~$\varphi$ is compact,
and arrive at
\[
\begin{split}
& \widehat{\Ttildeoeps g}(\xi)  
 = (2\pi {\rm i}  \xij) (2\pi {\rm i}  \xil)
   \int_{\varepsilon}^{\frac12}  \int_{\Rbbn} 
    \int_{\Rbbn} \varphi\left( y+ \frac{x-y}{t} \right)
        g(y) e^{-2 \pi {\rm i} x\cdot \xi} 
        \dx\dy \frac{\dt}{t^n} \\
& = (2\pi {\rm i}  \xij) (2\pi {\rm i}  \xil)
   \int_{\varepsilon}^{\frac12}  \int_{\Rbbn} 
    \int_{\Rbbn} \varphi\left( z \right)
        e^{-2 \pi {\rm i} (t \ z +(1-t) y)\cdot \xi}  g(y)
        \dz \dy \dt \\
& = (2\pi {\rm i}  \xij) (2\pi {\rm i}  \xil)
   \int_{\varepsilon}^{\frac12}  \int_{\Rbbn} 
    \varphihat\left( t \ \xi \right)
        e^{-2 \pi {\rm i} (1-t) y \cdot \xi}  g(y)
        \dy \dt \\
& = (2\pi {\rm i} \xij) (2\pi {\rm i} \xil)
   \int_{\varepsilon}^{\frac12}
    \varphihat\left( t \ \xi \right)
        \ghat((1-t)\xi) \dt \\
& = (2\pi {\rm i}  \xij) 
   \int_{\varepsilon}^{\frac12}
    \varphihat\left( t \ \xi \right)
    [2\pi {\rm i} (1-t) \xil] \ghat((1-t)\xi) \dt
    + (2\pi {\rm i}  \xij) 
   \int_{\varepsilon}^{\frac12}
    \varphihat\left( t \ \xi \right)
    [2\pi {\rm i} t \xil] \ghat((1-t)\xi) \dt \\
& = (2\pi {\rm i}  \xij) 
   \int_{\varepsilon}^{\frac12}
    \varphihat\left( t \ \xi \right)
    \widehat{\partialxl g}((1-t)\xi) \dt
    + (2\pi {\rm i}  \xij) 
   \int_{\varepsilon}^{\frac12}
    \widehat{\partialxl \varphi} (t \xi)
    \ghat((1-t)\xi) \dt .
\end{split}
\]
The assertion follows taking the limit $\varepsilon\to 0$.
\end{proof}

Next, we prove a technical result.
\begin{lemma} \label{lemma:2-Ttildeo}
Let~$\varphi$ be any of the two options in~\eqref{options-varphi}.
Then, the two following inequalities hold true:
\begin{subequations} \label{2.3Duran}
\begin{equation} \label{2.3DuranA}
2\pi \SemiNorm{\xij} \int_0^\infty \SemiNorm{\varphihat(t\xi)} \dt
\le \Cphirhoz
:= \rho^{-1} \Norm{\varphi}_{L^1(\Rbbn)}
  + \rho \Norm{\partialxjj \varphi}_{L^1(\Rbbn)}
  \quad \forall j=1,\dots,n,
\end{equation}
\begin{equation} \label{2.3DuranB}
2\pi \SemiNorm{\xij} \int_0^\infty 
        \SemiNorm{\widehat{\partialxl \varphi}(t\xi)} \dt
\le \Cphirhoo
:= \rho^{-1} \Norm{\partialxl \varphi}_{L^1(\Rbbn)}
  + \rho \Norm{\partialxjjl \varphi}_{L^1(\Rbbn)}
  \quad \forall j=1,\dots,n .
\end{equation}
\end{subequations}
\end{lemma}
\begin{proof}
The proof of~\eqref{2.3DuranA} is given in~\cite[Lemma~2.3]{Duran:2012}
and is therefore omitted here.
On the other hand, inequality~\eqref{2.3DuranB}
may be shown as an application of~\eqref{2.3DuranA}
to $\partialxil \varphi$.
\end{proof}

We are now in a position to prove the continuity
of the operator~$\Ttildeo$.
\begin{proposition} \label{proposition:Ttildeo}
Let~$\varphi$ be any of the two options in~\eqref{options-varphi}.
Then, for all~$g$ in $H^1(\Rbb^n)$,
the operator~$\Ttildeo$ defined in~\eqref{def:Ttildeo}
satisfies the following continuity property:
\[
\Norm{\Ttildeo g}_{0,\Rbbn}
\le 2^{\frac{n-1}{2}}
    \left[ \Cphirhoz \Norm{\partialxl g}_{0,\Rbbn}
      + \Cphirhoo \Norm{g}_{0,\Rbbn} \right] .
\]
If~$g$ vanishes outside~$\Omega$, we also have
\[
\Norm{\Ttildeo g}_{0,\Omega}
\le 2^{\frac{n-1}{2}}
    \left[ \Cphirhoz \Norm{\partialxl g}_{0,\Omega}
      + \Cphirhoo \Norm{g}_{0,\Omega} \right] .
\]
\end{proposition}
\begin{proof}
We only prove the second assertion
and focus on functions~$g$ in~$\Ccal^\infty_0(\Omega)$;
the general statement follows then from a density argument.

We consider splitting~\eqref{splitting:Ttildeoo+Ttildeot}
and show separate bounds for the two terms on the right-hand side.
The first one can be handled as in \cite[Lemma 2.4]{Duran:2012}
and its proof is therefore omitted:
\begin{equation} \label{need-of-isometry-1}
\Norm{\widehat{\Ttildeoo g}}_{0,\Rbbn}
\le 2^{\frac{n-1}{2}} \Cphirhoz \Norm{\widehat{\partialxl g}}_{0,\Rbbn}.
\end{equation}
Thus, we focus on the second term.
By the definition of~$\widehat{\Ttildeot g}$,
the Cauchy-Schwarz inequality implies
\[
\SemiNorm{\widehat{\Ttildeot g} (\xi)}^2
\le \left( \int_0^{\frac12} 2\pi \SemiNorm{\xij}
        \SemiNorm{\widehat{\partialxl \varphi}(t \xi)} \dt \right) 
\left( \int_0^{\frac12} 2\pi \SemiNorm{\xij}
        \SemiNorm{\widehat{\partialxl \varphi}(t \xi)}
        \SemiNorm{\ghat((1-t)\xi)}^2 \dt \right). 
\]
Using~\eqref{2.3DuranB}, we deduce
\[
\SemiNorm{\widehat{\Ttildeot g} (\xi)}^2
\le \Cphirhoo
\int_0^{\frac12} 2\pi \SemiNorm{\xij}
        \SemiNorm{\widehat{\partialxl \varphi}(t \xi)}
        \SemiNorm{\ghat((1-t)\xi)}^2 \dt .
\]
Integrating over~$\xi$ and employing
the change of variable $\eta=(1-t)\xi$ give
\[
\int_{\Rbbn} \SemiNorm{\widehat{\Ttildeot g} (\xi)}^2 \dxi
\le \Cphirhoo
\int_0^{\frac12}  \int_{\Rbbn} 
\frac{2\pi}{(1-t)^{n+1}} \SemiNorm{\etaj}
        \SemiNorm{\widehat{\partialxl \varphi}
        \left( \frac{t \eta}{1-t} \right)}
        \SemiNorm{\ghat(\eta)}^2 \deta \dt 
\]
If we consider the change of variable $s=t/(1-t)$, which entails
\[
\dt = (1+s)^{-2} \ds,\qquad
\frac{1}{(1-t)^{n+1}} = 
\left( \frac st \right)^{n+1}
= (1+s)^{n+1},
\]
then we arrive at
\[
\int_{\Rbbn} \SemiNorm{\widehat{\Ttildeot g} (\xi)}^2 \dxi
\le 2^{n-1} \Cphirhoo 
 \int_{\Rbbn} 
 \left( \int_0^1 2\pi \SemiNorm{\etaj} 
         \SemiNorm{\widehat{\partialxl \varphi}(s\eta)} \ds \right)
 \SemiNorm{\ghat(\eta)}^2 \deta  .
\]
We apply again~\eqref{2.3DuranB}:
\begin{equation} \label{need-of-isometry-2}
\Norm{\widehat{\Ttildeot g}}_{0,\Rbbn}^2
= \int_{\Rbbn} \SemiNorm{\widehat{\Ttildeot g} (\xi)}^2 \dxi
\le 2^{n-1} \Cphirhoo^2
\int_{\Rbbn} \SemiNorm{\ghat(\eta)}^2 \deta
= 2^{n-1} \Cphirhoo^2 \Norm{\ghat}_{0,\Rbbn}^2.
\end{equation}
The assertion follows using the Fourier isometry~\eqref{isometry:Fourier}
in~\eqref{need-of-isometry-1}, identity~\eqref{need-of-isometry-2},
and the properties of~$\omega$ detailed in Section~\ref{subsection:preliminary-results}.
\end{proof}

\subsection{Continuity of~$\Ttildet$}
\label{subsection:continuity-Ttildet}
We discuss the continuity of the operator in~\eqref{def:Ttildet}.
As discussed in~\cite{Duran:2012},
a direct application of the H\"older and Cauchy-Schwarz inequalities
would end up with suboptimal estimates
as those in~\cite{Galdi:2011}.
Therefore, finer estimates are in order.
To this aim, we extend \cite[Section 3]{Duran:2012} to the first order case.



\begin{proposition} \label{proposition:Ttildet}
Let~$g$ and~$\Ttildet$ be as in~\eqref{options-varphi} and~\eqref{def:Ttildet}.
Assume that~$g$ belongs to~$L^2(\Rbbn)$ and has support contained in~$\Omega$.
Given $1\le p < n/(n-1)$
and~$p'$ the conjugate index of~$p$,
the following inequality holds true:
\[
\Norm{\Ttildet g}_{0,\Omega}
\le \frac{2^{\frac{n}{2}}}{(1-\frac{n}{p'})^{\frac{p}{2}}}
    \SemiNorm{\Omega}^{1-\frac p2}
  \Norm{\partialxjl \varphi}_{L^1(\Omega)}^{\frac{p}{2}}
  \Norm{\partialxjl \varphi}_{L^\infty(\Omega)}^{1-\frac{p}{2}}
  \Norm{g}_{0,\Omega}.
\]
\end{proposition}
\begin{proof}
The proof follows along the same lines of
\cite[Lemma~3.2]{Duran:2012},
the only difference being the number of derivatives of~$\varphi$.
\end{proof}

\subsection{Continuity of~$\Ttilde$}
\label{subsection:continuity-Ttilde}
We prove the continuity of the operator in~\eqref{def:Ttilde}.

\begin{theorem} \label{theorem:continuity-Ttilde}
Let~$g$ and~$\Ttilde$ be as in~\eqref{options-varphi} and~\eqref{def:Ttilde}.
Assume that~$g$ belongs to~$H^1_0(\Omega)$.
Given $1\le p < n/(n-1)$
and~$p'$ the conjugate index of~$p$,
the following inequality holds true:
\begin{equation} \label{continuity-TtildeA}
\begin{split}
\Norm{\Ttilde g}_{0,\Omega}
& \le 2^{\frac{n-1}{2}}
     (\rho^{-1} \Norm{\varphi}_{L^1(\Omega)}
        + \rho \Norm{\partialxjj \varphi}_{L^1(\Omega)})
     \Norm{\partialxl g}_{0,\Omega} \\
& \quad + 2^{\frac{n-1}{2}}
    (\rho^{-1} \Norm{\partialxl \varphi}_{L^1(\Omega)}
        + \rho \Norm{\partialxjjl \varphi}_{L^1(\Omega)})
    \Norm{g}_{0,\Omega} \\
&\quad + \frac{2^{\frac{n}{2}}}{(1-\frac{n}{p'})^{\frac{p}{2}}}
  \SemiNorm{\Omega}^{1-\frac p2}
  \Norm{\partialxjl \varphi}_{L^1(\Omega)}^{\frac{p}{2}}
  \Norm{\partialxjl \varphi}_{L^\infty(\Omega)}^{1-\frac{p}{2}}
  \Norm{g}_{0,\Omega}.
\end{split}
\end{equation}
\end{theorem}
\begin{proof}
The assertion follows combining
Lemmas~\ref{proposition:Ttildeo} and~\ref{proposition:Ttildet},
and the explicit representation of the constants
$\Cphirhoz$ and~$\Cphirhoo$
in~\eqref{2.3Duran}.
\end{proof} 

Next, we derive explicit constants with respect to~$\hOmega$ and~$\rho$
for inequality~\eqref{continuity-TtildeA},
i.e., we are in a position for proving one of the main results of the manuscript.

\begin{proof}[Proof of Theorem~\ref{theorem:1st-Babu-Aziz}.]
For all $j$, $\ell=1,\dots,n$,
we have to bound the two terms on the right-hand side
of splitting~\eqref{splitting:Ttilde-2nd-derivative}.
We only prove bounds for~$\Ttildekjlt (\yk f(y))$,
as the bounds for~$\Ttildekjlo (f(y))$ are analogous.
For the term~$\Ttildekjlo (\yk f(y))$, we have that
$\varphi(x)$ in~\eqref{options-varphi}
is given by $\omega(x)$,
which is supported, with integral~$1$,
in the ball~$\Brho$ of radius~$\rho$.
Without loss of generality, the ball can be centred at the origin.

We can write
\[
\varphi(x)
= \rho^{-n}  \psi(\rho^{-1} x),
\]
where~$\psi$ has the same smoothness as~$\varphi$,
is supported in the unitary ball~$\Boz$
with integral~$1$,
and is fixed once and for all.

The following properties of~$\varphi$ and its derivatives are valid:
\[
\begin{alignedat}{3}
& \quad
&& \Norm{\varphi}_{L^1(\Rbbn)} \approx 1,\quad
&& \Norm{\varphi}_{L^\infty(\Rbbn)} \approx \rho^{-n};\\
& \partialxj \varphi(x) = \rho^{-n-1}\partialxj \varphi(\rho^{-1}x),\quad
\qquad
&& \Norm{\partialxj\varphi}_{L^1(\Rbbn)} \approx \rho^{-1},\quad
&& \Norm{\partialxj\varphi}_{L^\infty(\Rbbn)} \approx \rho^{-n-1};\\
& \partialxjl \varphi(x) = \rho^{-n-2}\partialxjl \varphi(\rho^{-1}x),\quad
&& \Norm{\partialxjl\varphi}_{L^1(\Rbbn)} \approx \rho^{-2},\quad
&& \Norm{\partialxjl\varphi}_{L^\infty(\Rbbn)} \approx \rho^{-n-2};\\
& \partialxjjl \varphi(x) = \rho^{-n-3}\partialxjjl \varphi(\rho^{-1}x),\quad
&& \Norm{\partialxjjl\varphi}_{L^1(\Rbbn)} \approx \rho^{-3},\quad
&& \Norm{\partialxjjl\varphi}_{L^\infty(\Rbbn)} \approx \rho^{-n-3}.
\end{alignedat}
\]
The definition of $\Ttildekjlt(\yk f)$,
the chain rule,
the inequality~$\vert \yk \vert \le R$,
and the fact that $\rho^{-1} \le R\rho^{-2}$ imply
\small{\begin{equation} \label{eq-green}
\begin{split}
\SemiNorm{\ubf}_{2,\Omega}
& \lesssim_n
  \hOmega \left[\rho^{-1} +\rho \ \rho^{-2} \right] \SemiNorm{f}_{1,\Omega}
  + \left[\rho^{-1} + \rho \rho^{-2}\right] \Norm{f}_{0,\Omega}\\
& \qquad\qquad\qquad
    + \hOmega [\rho^{-1}\rho^{-1} + \rho \rho^{-3}] \Norm{f}_{0,\Omega}
    + \hOmega \frac{2^{\frac n2}}
                    {\left(1-\frac{n}{p'}\right)^{\frac p2}}
  \SemiNorm{\Omega}^{1-\frac p2} (\rho^{-2})^{\frac p2}
  (\rho^{-n-2})^{1-\frac p2} \Norm{f}_{0,\Omega}\\
& \lesssim  \hOmega \left[ \rho^{-1} \SemiNorm{f}_{1,\Omega}
        + \rho^{-2} \Norm{f}_{0,\Omega}
        + \frac{2^{\frac n2}}
         {\left(1-\frac{n}{p'}\right)^{\frac p2}}
        \SemiNorm{\Omega}^{1-\frac p2}
        \rho^{-2-n\left( 1-\frac p2 \right)}
        \Norm{f}_{0,\Omega} \right]\\
& \lesssim_n
    \frac{\hOmega}{\rho} \SemiNorm{f}_{1,\Omega}
    + \frac{\hOmega}{\rho^2} \Norm{f}_{0,\Omega}
        \left[ 1 + \left(1-\frac{n}{p'}\right)^{-\frac p2}
        \SemiNorm{\Omega}^{1-\frac p2}
        \rho^{-n\left( 1-\frac p2 \right)}
            \right] .
\end{split}
\end{equation}}\normalsize
We cope with part of the last coefficient on the right-hand side separately.
Using that
\[
\SemiNorm{\Brho}^{1-\frac p2} \rho^{-n\left( 1-\frac p2 \right)} \approx 1,
\]
we can write
\begin{equation} \label{eq-heart}
\begin{split}
& \left(1-\frac{n}{p'}\right)^{-\frac p2}
   \SemiNorm{\Omega}^{1-\frac p2}
   \rho^{-n\left( 1-\frac p2 \right)}
= \left(1-\frac{n}{p'}\right)^{-\frac p2}
   \left(\frac{\SemiNorm{\Omega}}{\SemiNorm{\Brho}}\right)^{1-\frac p2}
   \SemiNorm{\Brho}^{1-\frac p2}
   \rho^{-n\left( 1-\frac p2 \right)}\\
& \approx \left(1-\frac{n}{p'}\right)^{-\frac p2}
   \left(\frac{\SemiNorm{\Omega}}{\SemiNorm{\Brho}}\right)^{\frac{n-2}{2(n-1)}}
   \left(\frac{\SemiNorm{\Omega}}{\SemiNorm{\Brho}}\right)^{\frac12 \left( \frac{n}{n-1} - p \right) }.
\end{split}
\end{equation}
For $\SemiNorm{\Omega}/\SemiNorm{\Brho}$ sufficiently large
(the ball~$\Brho$ is anyhow fixed once and for all in the reference framework),
we choose~$p$ such that
\[
\frac12 \left( \frac{n}{n-1} - p \right) = \frac{1}{\log\left( \frac{\SemiNorm{\Omega}}{\SemiNorm{\Brho}} \right)}.
\]
Equivalently, we pick~$p$ such that
\[
\log\left( \frac{\SemiNorm{\Omega}}{\SemiNorm{\Brho}} \right)
= \frac{1}{\frac12 \left( \frac{n}{n-1} - p \right)}
\quad\Longrightarrow\quad
\frac{\SemiNorm{\Omega}}{\SemiNorm{\Brho}}
= e^{\frac{1}{\frac12 \left( \frac{n}{n-1} - p \right)}}
\quad\Longrightarrow\quad
\left( \frac{\SemiNorm{\Omega}}{\SemiNorm{\Brho}} \right)
^{\frac12 \left( \frac{n}{n-1} - p \right)} = e^1.
\]
We plug the last identity in~\eqref{eq-heart} and deduce
\[
\left(1-\frac{n}{p'}\right)^{-\frac p2}
   \SemiNorm{\Omega}^{1-\frac p2}
   \rho^{-n\left( 1-\frac p2 \right)}
\approx 
 \left(1-\frac{n}{p'}\right)^{-\frac p2}
   \left(\frac{\SemiNorm{\Omega}}{\SemiNorm{\Brho}}\right)^{\frac{n-2}{2(n-1)}} e^1.
\]
Going back to~\eqref{eq-green}, we write
\[
\SemiNorm{\ubf}_{2,\Omega}
 \lesssim_n
 \frac{\hOmega}{\rho} \SemiNorm{f}_{1,\Omega}
    + \frac{\hOmega}{\rho^2} \Norm{f}_{0,\Omega}
        \left[ 1 + \left(1-\frac{n}{p'}\right)^{-\frac p2}
   \left(\frac{\SemiNorm{\Omega}}{\SemiNorm{\Brho}}\right)^{\frac{n-2}{2(n-1)}}
            \right] .
\]
We further note that
\[
1-\frac{n}{p'}
= \left( \frac{n-1}{p} \right) \left( \frac{n}{n-1} - p \right)
= \frac{2(n-1)}{p \log\left( \frac{\SemiNorm{\Omega}}{\SemiNorm{\Brho}} \right)}.
\]
We combine the two above displays:
\[
\SemiNorm{\ubf}_{2,\Omega}
 \lesssim_n
 \frac{\hOmega}{\rho} \SemiNorm{f}_{1,\Omega}
    + \frac{\hOmega}{\rho^2} \Norm{f}_{0,\Omega}
        \left[ 1 +  \left(\frac{\SemiNorm{\Omega}}{\SemiNorm{\Brho}}\right)^{\frac{n-2}{2(n-1)}}
            \frac{p^{\frac p2}}{2(n-1)^{\frac p2}} 
            \log\left( \frac{\SemiNorm{\Omega}}{\SemiNorm{\Brho}} \right)^{\frac p2} \right] .
\]
Using $p < n/(n-1)$, the following quantity is uniformly bounded in~$n$ and thus in~$p$:
\[
\frac{p^{\frac p2}}{2(n-1)^{\frac p2}} .
\]
Moreover, we know that $p/2 < n/(2(n-1))$.
We deduce that
\[
\SemiNorm{\ubf}_{2,\Omega}
 \lesssim_n
 \frac{\hOmega}{\rho} \SemiNorm{f}_{1,\Omega}
    + \frac{\hOmega}{\rho^2} \Norm{f}_{0,\Omega}
        \left[ 1 +  \left(\frac{\SemiNorm{\Omega}}{\SemiNorm{\Brho}}\right)^{\frac{n-2}{2(n-1)}}
            \log\left( \frac{\SemiNorm{\Omega}}{\SemiNorm{\Brho}} \right)^{\frac{n}{2(n-1)}} \right] ,
\]
which is the assertion.

\end{proof}

Compared to the lowest order estimate~\eqref{explicit-estimates-Duran},
we have an extra term involving the gradient of~$f$
and an extra $\rho^{-1}$ scaling factor for the term involving~$f$.

\begin{remark} \label{remark:union-star-shaped}
The issue on whether estimates as in Theorem~\ref{theorem:1st-Babu-Aziz}
can be extended to union of star-shaped domains
was addressed in~\cite{Guzman-Salgado:2021}.
Their proof relies on partition of unity techniques;
this entails that estimates have constants that are not fully explicit
with respect to the shape of the domain~\cite{Guzman:2024}.
A simpler open problem is whether
one may be able to prove
arbitrary order Babu\v ska-Aziz inequalities
with explicit constants
on the union of simpler star-shaped domains,
e.g., on simplicial patches.
\end{remark}

\begin{remark} \label{remark:boundary-conditions}
We have that~$\ubf$ in~\eqref{eq:bogovskii}
satisfies the boundary conditions,
i.e., that~$\ubf$ belongs to~$[H_0^2(\Omega)]^n$;
this is shown for instance in \cite[Lemma III.3.1]{Galdi:2011} for smooth functions;
the corresponding result for Sobolev spaces is achieved via density arguments.
As it is of independent interest,
we discuss an alternative proof
of this fact in Appendix~\ref{appendix:alternative-zero-trace},
which holds also in the non-Hilbertian setting.
\end{remark}

\subsection{On the optimality of the estimates in Theorem~\ref{theorem:1st-Babu-Aziz}
in 2D}
\label{subsection:counterexample}

We comment on the optimality of the estimates
in Theorem~\ref{theorem:1st-Babu-Aziz}
for planar domains
based on a counterexample in \cite[Section~3]{Duran:2012}.
Introduce the domain
\[
\Omegaaeps:= (-a,a) \times (-\varepsilon, \varepsilon)
\]
and the function
\[
f(x_1,x_2)=x_1.
\]
Let~$\ubf$ be the solution to the divergence problem~\eqref{1st-Babu-Aziz}. We have
\[
\Norm{x_1}_{0,\Omegaaeps}^2
 = \int_{\Omegaaeps} x_1 \div\ubf
  = -\int_{\Omegaaeps} \ubf_1
  = -\frac12 \int_{\Omegaaeps} x_2^2 \partial_{x_2^2}^2 \ubf_1
\le \frac12 \Norm{x_2^2}_{0,\Omegaaeps}
    \Norm{\partial_{x_2^2}^2 \ubf_1}_{0,\Omegaaeps}.
\]
We use estimates as in~\eqref{1st-Babu-Aziz} for the last term on the right-hand side:
there exists positive constants~$C_1$ and~$C_2$ depending
on~$\hOmega=2a$ and~$\rho=\varepsilon$ such that
\[
\Norm{x_1}_{0,\Omegaaeps}^2
\le \frac12 \Norm{x_2^2}_{0,\Omegaaeps}
    \left[  \CBAoA \Norm{x_1}_{0,\Omegaaeps}
      + \CBAoB \Norm{1}_{0,\Omegaaeps} \right] .
\]
We have
\[
 \Norm{x_1}_{0,\Omegaaeps}^2   
    = \frac43 a^3 \varepsilon ,
    \qquad\qquad
 \Norm{x_2^2}_{0,\Omegaaeps}^2
    = \frac45 a \varepsilon^5 ,
    \qquad\qquad
 \Norm{1}_{0,\Omegaaeps}^2     
 = 4 a \varepsilon.
\]
Combining the two displays, we get
\[
a^3 \varepsilon
\lesssim  \CBAoA a^{2} \varepsilon^{3}
        + \CBAoB a \varepsilon^{3},
\]
whence
\[
1
\lesssim \CBAoA a^{-1} \varepsilon^2 + \CBAoB a^{-2} \varepsilon^2 
\approx \CBAoA \frac{\rho^2}{\hOmega} + \CBAoB \frac{\rho^2}{\hOmega^2} .
\]
This inequality implies that
at least one of the following must hold true:
\[
\CBAoA \gtrsim \frac{\hOmega}{\rho^2},
\qquad\qquad\qquad
\CBAoB \gtrsim \frac{\hOmega^2}{\rho^2}.
\]
Using~\eqref{estimates:Bogo-constants},
we deduce
\[
\frac{\hOmega^2}{\rho^2}
\lesssim \CBAoB 
\lesssim \frac{\hOmega}{\rho},
\]
which cannot be valid in general with hidden constants
independent of~$R$ and~$\rho$ as in~\eqref{Rrho}.
This entails that
\[
\CBAoA \gtrsim \frac{\hOmega}{\rho^2},
\]
i.e., the first bound in~\eqref{estimates:Bogo-constants}
is optimal up to a logarithmic factor.

We are not able to infer a clear statement on the optimality of~$\CBAoB$ from the bound above.
A heuristic argument based on scaling  techniques, suggests
however that bound~\eqref{estimates:Bogo-constants} on~$\CBAoB$ should be also optimal.

\section{Proof of a Ne\v cas-Lions inequality with symmetric gradients} \label{section:NL-sg}

We prove Theorem~\ref{theorem:new-Necas-Lions-symmetric-gradients}.
To this aim, we first show an auxiliary result.
\begin{lemma} \label{lemma:gradcurl-2curlgradsym}
Let~$\ubf$ be a sufficiently smooth vector field.
Then, the following identities hold true:
\begin{equation} \label{gradcurl-2curlgradsym}
[\nabla(\nabla\times\ubf)]^T
= \nabla\times(\nabla \ubf + (\nabla \ubf)^T)
= 2 \nabla\times(\nablaS \ubf),
\end{equation}
where, for a given tensor~$\mathbf A$,
$\nabla\times \mathbf A$
denotes the matrix obtained by
applying the curl operator row-wise to~$\mathbf A$.
\end{lemma}
\begin{proof}
Direct computations give
\[
[\nabla(\nabla\times\ubf)]^T =
\begin{bmatrix}
\partial_{x_1}(\partial_{x_2}\ubf_3 - \partial_{x_3}\ubf_2)
& \partial_{x_1}(\partial_{x_3}\ubf_1 - \partial_{x_1}\ubf_3)
& \partial_{x_1}(\partial_{x_1}\ubf_2 - \partial_{x_2}\ubf_1) \\
\partial_{x_2}(\partial_{x_2}\ubf_3 - \partial_{x_3}\ubf_2)
& \partial_{x_2}(\partial_{x_3}\ubf_1 - \partial_{x_1}\ubf_3)
& \partial_{x_2}(\partial_{x_1}\ubf_2 - \partial_{x_2}\ubf_1) \\
\partial_{x_3}(\partial_{x_2}\ubf_3 - \partial_{x_3}\ubf_2)
& \partial_{x_3}(\partial_{x_3}\ubf_1 - \partial_{x_1}\ubf_3)
& \partial_{x_3}(\partial_{x_1}\ubf_2 - \partial_{x_2}\ubf_1)
\end{bmatrix}.
\]
Similarly, we may show that the above tensor coincides with
$\nabla\times(\nabla \ubf + (\nabla \ubf)^T)$,
which gives the first identity in~\eqref{gradcurl-2curlgradsym};
the second identity in~\eqref{gradcurl-2curlgradsym}
is a consequence of the definition of the symmetric gradient.
\end{proof}
We are now in a position to prove Theorem~\ref{theorem:new-Necas-Lions-symmetric-gradients}.
For any generic~$\qbfRM$ in the space
of rigid body motions~$\RM(\Omega)\cap [L^2_0(\Omega)]^n$
as discussed in Section~\ref{subsection:main-result3},
we have
\[
\Norm{\vbf-\PiRM \vbf}_{0,\Omega}
\le \Norm{\vbf-\qbfRM - \Piboldzz(\vbf-\qbfRM)}_{0,\Omega}
\qquad\qquad
\forall \vbf \in [L^2(\Omega)]^n.
\]
An immediate consequence of the standard, vector version,
lowest order Ne\v cas-Lions
inequality~\eqref{0th-Necas-Lions} is that
\begin{equation} \label{addendum-1}
\Norm{\vbf-\PiRM \vbf}_{0,\Omega}
\le \CNLz \Norm{\nabla(\vbf-\qbfRM)}_{-1,\Omega}
\qquad\qquad \forall \qbfRM \in \RM(\Omega) \cap [L^2_0(\Omega)]^n.
\end{equation}
We are left to prove the existence of a positive constant~$C$
with explicit dependence on~$R$ and~$\rho$ as in~\eqref{Rrho},
such that,
for a specific choice of~$\qbfRM$,
the following inequality is valid:
\[
\Norm{\nabla(\vbf-\qbfRM)}_{-1,\Omega}
\le C \Norm{\nablaS \vbf}_{-1}.
\]
Using splitting~\eqref{splitting-gradient}
of the gradient into symmetric and skew-symmetric parts,
and the triangle inequality entails
\begin{equation} \label{addendum-2}
\Norm{\nabla(\vbf-\qbfRM)}_{-1,\Omega}
\le \Norm{\nablaS(\vbf-\qbfRM)}_{-1,\Omega}
     + \Norm{\nablaSS(\vbf-\qbfRM)}_{-1,\Omega}.
\end{equation}
Since~$\qbfRM$ is a rigid body motion,
$\nablaS \qbfRM$  is the zero tensor in~$\Rbb^{n\times n}$:
the first term on the right-hand side is equal to~$\Norm{\nablaS \vbf}_{-1,\Omega}$.
As for the second term on the right-hand side,
we define~$A^{n\times n}$ as the space of $(n\times n)$ skew-symmetric matrices,
$n=2,3$.
We take~$\qbfRM$ such that
\[
\Norm{\nablaSS(\vbf-\qbfRM)}_{-1,\Omega}
: = \inf_{\widetilde\qbfRM \in \RM(\Omega)\cap[L^2_0(\Omega)]^{n\times n}} \Norm{\nablaSS(\vbf-\widetilde\qbfRM)}_{-1,\Omega}
= \inf_{\cbf \in A^{n\times n}} \Norm{\nablaSS \vbf- \cbf}_{-1,\Omega} .
\]
With this at hand, elementary computations give
\begin{equation}
\label{skewviacurl}
\Norm{\nablaSS(\vbf-\qbfRM)}_{-1,\Omega}
= \inf_{\cbf \in A^{n\times n}} \Norm{\nablaSS\vbf-\cbf}_{-1,\Omega}
\le \frac1{\sqrt2} \inf_{\cbf \in \Rbb^{2n-3}} \Norm{\nabla\times\vbf - \cbf}_{-1,\Omega}.
\end{equation}
Note that the last two norms involve tensors and vectors, respectively.

An immediate consequence of
the scalar and vector versions of~\eqref{1st-Necas-Lions},
\eqref{Poincare},
and~\eqref{CBo-le-CNLo} is that
\[
\inf_{\cbf \in \Rbb^{2n-3}} \Norm{\nabla\times\vbf - \cbf}_{-1,\Omega}
\le \CNLo \Norm{\nabla(\nabla\times\vbf)}_{-2,\Omega}
\le (\CBAoA\CP R+\CBAoB) \Norm{\nabla(\nabla\times\vbf)}_{-2,\Omega}.
\]
Lemma~\ref{lemma:gradcurl-2curlgradsym} reveals that
\begin{equation} \label{dependence-CNLzstar}
\begin{split}
\inf_{\cbf \in \Rbb^{2n-3}} \Norm{\nabla\times\vbf - \cbf}_{-1,\Omega}
& \le 2 (\CBAoA\CP R+\CBAoB) 
    \Norm{\boldsymbol{\nabla} \boldsymbol{\times} \nablaS \vbf}_{-2,\Omega} \\
& \le 2 (\CBAoA\CP R+\CBAoB) \Norm{\nablaS \vbf}_{-1,\Omega}.
\end{split}
\end{equation}
Combining~\eqref{addendum-1}, \eqref{addendum-2},
\eqref{skewviacurl}, and~\eqref{dependence-CNLzstar}
yields the assertion.

\paragraph*{Acknowledgements.}
The Authors have been partially funded by the European Union
(ERC, NEMESIS, project number 101115663);
views and opinions expressed are however those of the author(s) only
and do not necessarily reflect those of the EU or the ERC Executive Agency.
LM has been partially funded by MUR (PRIN2022 research grant n. 202292JW3F).
The Authors are also members of the Gruppo Nazionale Calcolo Scientifico-Istituto
Nazionale di Alta Matematica (GNCS-INdAM).
The authors are grateful to an anonymous referee
for the helpful comments that lead
to an improved version of the manuscript.

{\footnotesize
\bibliography{bibliography.bib}}
\bibliographystyle{plain}

\appendix
\section{Arbitrary order Babu\v ska-Aziz inequalities
on star-shaped domains with explicit constants}
\label{appendix:general-order-BA}

In this appendix, we provide a road-map
for proving arbitrary order Babu\v ska-Aziz inequalities
that are explicit with respect
to~$R$ and~$\rho$ in~\eqref{Rrho},
based on a generalisation of the analysis in Section~\ref{section:2nd-Bogovskii}.
In particular, we prove the following generalisation of Theorem~\ref{theorem:1st-Babu-Aziz}.

\begin{theorem} \label{theorem:general-order-Babu-Aziz}
Let~$d$ in~$\Nbb$ be larger than or equal to~$1$,
and~$R$ and~$\rho$ be as in~\eqref{Rrho}.
Assume that~$f$ is in~$H^{d-1}_0(\Omega)\cap L^2_0(\Omega)$.
Then, there exists~$\ubf$ in $[H^d_0(\Omega)]^n$ such that
$\div \ubf=f$ and
\begin{equation} \label{explicit-general-order-estimates}
\begin{split}
\SemiNorm{\ubf}_{d,\Omega}
& \lesssim_n \sum_{\ell=1}^{d-1}
            \frac{\hOmega}{\rho^{d-\ell}}
            \SemiNorm{f}_{\ell,\Omega} 
        + \frac{\hOmega}{\rho^d} 
            \left[ 1 + \left( \frac{\SemiNorm{\Omega}}{\SemiNorm{\Brho}} \right)^{\frac{n-2}{2(n-1)}}
            \left( \log\frac{\SemiNorm{\Omega}}{\SemiNorm{\Brho}} \right)^{\frac{n}{2(n-1)}}
            \right] \Norm{f}_{0,\Omega}.
\end{split}
\end{equation}
\end{theorem}

We give some details on how to prove~\eqref{explicit-general-order-estimates}.
Essentially, we have to prove the continuity of the counterpart 
of the operator in~\eqref{def:Ttilde}
involving all the derivatives of order~$d$.
More precisely,
for each multi-index $\jbf=(j_1, \dots,j_n)$
in $\{ 1 ,\dots, d \}^n$,
$\SemiNorm{\jbf}=d$,
we study the continuity of the operator
\begin{equation} \label{def:Ttilde-general}
\Ttilde g
:= \Ttildeo g + \Ttildet g,
\end{equation}
where
\begin{equation} \label{def:Ttildeo-general}
\Ttildeo g(x)
:= \lim_{\varepsilon\to0}
  \int_{\varepsilon}^{\frac12} \int_{\Rbbn}
        \partialxjoxjn
        \left[ \varphi\left( y+ \frac{x-y}{t} \right) \right]
        g(y) \dy \frac{\dt}{t^n}
\end{equation}
and
\begin{equation} \label{def:Ttildet-general}
\Ttildet f(x)
:= \int_{\frac12}^1 \int_{\Rbbn} 
        \partialxjoxjn
        \left[ \varphi\left( y+ \frac{x-y}{t} \right) \right]
        g(y) \dy \frac{\dt}{t^n}.
\end{equation}
Above, $g$ and~$\varphi$ are precisely as in~\eqref{options-varphi}.
The general assertion then follows summing over
all possible multi-indices~$\jbf$.
In the remainder of the appendix, 
\begin{equation} \label{jbf-fixed}
\jbf \text{ is fixed,}
\qquad\qquad\qquad
j_1\ne 0.
\end{equation}

\subsection{Continuity of~$\Ttildeo$}
\label{subsection:continuity-Ttildeo-general}

We discuss here the continuity of the operator in~\eqref{def:Ttildeo-general}.
We have the two following technical results. 

\begin{lemma} \label{lemma:1-Ttilde1-general}
Let~$\jbf$ be as in~\eqref{jbf-fixed}.
If~$g$ in~\eqref{options-varphi} belongs to~$\Ccalzinf(\Rbbn)$,
then the following identity is valid:
\small{\[
\begin{split}
\widehat{\Ttildeo g}(\xi)
& = (2\pi {\rm i}  \xio)
   \lim_{\varepsilon\to 0} 
   \left[
   \sum_{\kbf,\ellbf \in \{ 1,\dots,d \}^n ,\kbf+\ellbf=(j_1-1,j_2,\dots,j_d) }
   \int_{\varepsilon}^{\frac12}
   \widehat{\partial^{\SemiNorm{\kbf}}_{x_1^{k_1},\dots,x_n^{k_n}} \varphi} (t \xi)
    \widehat{\partial^{\SemiNorm{\ellbf}}_{x_1^{\ell_1},\dots,x_n^{\ell_n}} g}(\xi) \dt 
     \right].
\end{split}
\]}\normalsize
\end{lemma}
\begin{proof}
The proof is a modification of that of Lemma~\ref{lemma:1-Ttilde1}.
Since this result contains most of the differences
compared to the first order case,
we prove the assertion for the second order case,
i.e., we assume that~$\jbf=(\jo,\jtw,\jth)$
with $\jo+\jtw+\jth=3$.
To further simplify the proof, we discuss the case $j_1=j_2=j_3=1$.
The general assertion is proven analogously.

In analogy to the proof of Lemma~\ref{lemma:1-Ttilde1}, we can write
\[
\begin{split}
& \widehat{\Ttildeoeps g}(\xi)  
 = (2\pi {\rm i}  \xio) (2\pi {\rm i}  \xitw)(2\pi {\rm i}  \xith)
   \int_{\varepsilon}^{\frac12}  \int_{\Rbbn} 
    \int_{\Rbbn} \varphi\left( y+ \frac{x-y}{t} \right)
        g(y) e^{-2 \pi {\rm i} x\cdot \xi} 
        \dx\dy \frac{\dt}{t^n} \\
& = (2\pi {\rm i}  \xio) \left[
   \int_{\varepsilon}^{\frac12}
    \varphihat\left( t \ \xi \right)
    [2\pi {\rm i} (1-t) \xitw][2\pi {\rm i} (1-t) \xith] \ghat((1-t)\xi) \dt \right. \\
& \qquad\qquad\qquad  +   \int_{\varepsilon}^{\frac12}
    \varphihat\left( t \ \xi \right)
    [2\pi {\rm i} t \xitw][2\pi {\rm i} (1-t) \xith] \ghat((1-t)\xi) \dt \\
& \qquad\qquad\qquad +    \int_{\varepsilon}^{\frac12}
    \varphihat\left( t \ \xi \right)
    [2\pi {\rm i} (1-t) \xitw][2\pi {\rm i} t \xith] \ghat((1-t)\xi) \dt \\
& \qquad\qquad\qquad \left.   +   \int_{\varepsilon}^{\frac12}
    \varphihat\left( t \ \xi \right)
    [2\pi {\rm i} t \xitw][2\pi {\rm i} t \xith] \ghat((1-t)\xi) \dt \right] \\
& = (2\pi {\rm i}  \xio) \left[
   \int_{\varepsilon}^{\frac12}
    \varphihat\left( t \ \xi \right)
    \widehat{\partial^2_{\xtw\xth} g}((1-t)\xi) \dt
    + \int_{\varepsilon}^{\frac12}
    \widehat{\partial_{\xtw} \varphi} (t \xi)
    \widehat{ \partial_{\xth}g}((1-t)\xi) \dt \right. \\
& \qquad\qquad\qquad \left. + \int_{\varepsilon}^{\frac12}
    \widehat{\partial_{\xth} \varphi}\left( t \ \xi \right)
    \widehat{\partial_{\xtw} g}((1-t)\xi) \dt
    + \int_{\varepsilon}^{\frac12}
    \widehat{\partial^2_{\xtw \xth} \varphi} (t \xi)
    \ghat((1-t)\xi) \dt \right] .
\end{split}
\]
This yields the assertion for the case $d=3$, and $j_1=j_2=j_3=1$.
\end{proof}
From Lemma~\ref{lemma:1-Ttilde1-general}, it is apparent that we
have to bound the norm of several derivatives of~$\varphi$,
which is what we accomplish in the next result.

\begin{lemma} \label{lemma:2-Ttildeo-general}
Let~$\varphi$ and~$\jbf$ be as in~\eqref{options-varphi} and~\eqref{jbf-fixed}.
Then, the following inequalities hold true:
for all multi-indices~$\kbf$ in~$\{ 1,\dots,d \}^n$
such that~$\SemiNorm{\kbf} \le \SemiNorm{\jbf}=d$,
\small{\begin{equation} \label{2.3Duran-general}
\begin{split}
2\pi \SemiNorm{\xio} \int_0^\infty 
        \SemiNorm{\widehat{\partial^{\SemiNorm{\kbf}}_{\xo^{k_1-1},\dots,\xn^{k_n}} \varphi}(t\xi)} \dt
& \le \Cphirhokbf \\
& := \rho^{-1} \Norm{\partial^{\SemiNorm{\kbf}}_{x_1^{k_1-1},x_2^{k_2},\dots,x_n^{k_n}} \varphi}_{L^1(\Rbbn)}
  + \rho \Norm{\partial^{\SemiNorm{\kbf}+2}_{x_1^{k_1+1},x_2^{k_2},\dots,x_n^{k_n}} \varphi}_{L^1(\Rbbn)} .
  \end{split}
\end{equation}}\normalsize
\end{lemma}
\begin{proof}
The proof is a modification of that of Lemma~\ref{lemma:2-Ttildeo}.
\end{proof}

The two above technical lemmas give the following result.

\begin{proposition} \label{proposition:Ttildeo-general}
Let~$\varphi$ be any of the two options in~\eqref{options-varphi}.
Then, for all~$g$ in $H^d(\Rbb^n)$,
the operator~$\Ttildeo$ defined in~\eqref{def:Ttildeo-general}
satisfies the following continuity property:
\[
\Norm{\Ttildeo g}_{0,\Rbbn}
\le 2^{\frac{n-1}{2}}
     \sum_{\kbf,\ellbf \in \{ 1,\dots,d \}^n ,\kbf+\ellbf=(j_1-1,\dots,j_d) }
            \left[\Cphirhokbf
            \Norm{\partial^{\SemiNorm{\ellbf}}_{x_1^{\ell_1},\dots,x_n^{\ell_n}} g}_{0,\Rbbn} \right],
\]
where~$\varphi$ is any of the two options in~\eqref{options-varphi}.
\end{proposition}
\begin{proof}
The proof is a modification of that of Proposition~\ref{proposition:Ttildeo},
and further combines Lemmas~\ref{lemma:1-Ttilde1-general} and~\ref{lemma:2-Ttildeo-general}.
\end{proof}

\subsection{Continuity of~$\Ttildet$}
\label{subsection:continuity-Ttildet-general}

We discuss here the continuity of the operator in~\eqref{def:Ttildet-general}.
We have the following result.

\begin{proposition} \label{proposition:Ttildet-general}
Let~$g$, $\Ttildet$, and~$\jbf$ be as
in~\eqref{options-varphi}, \eqref{def:Ttildet-general}, and~\eqref{jbf-fixed}.
Assume that~$g$ belongs to~$L^2(\Rbbn)$ and has support contained in~$\Omega$.
Given $1\le p < n/(n-1)$
and~$p'$ the conjugate index of~$p$,
the following inequality holds true:
\[
\Norm{\Ttildet g}_{0,\Omega}
\le \frac{2^{\frac{n}{2}}}{(1-\frac{n}{p'})^{\frac{p}{2}}}
    \SemiNorm{\Omega}^{1-\frac p2}
  \Norm{\partial^d_{x_{1}^{j_1}\dots x_n^{j_n}} \varphi}_{L^1(\Omega)}^{\frac{p}{2}}
  \Norm{\partial^d_{x_{1}^{j_1}\dots x_n^{j_n}} \varphi}_{L^\infty(\Omega)}^{1-\frac{p}{2}}
  \Norm{g}_{0,\Omega}.
\]
\end{proposition}
\begin{proof}
The proof is a modification of that of Proposition~\ref{proposition:Ttildet}.
\end{proof}

\subsection{Continuity of~$\Ttilde$}
\label{subsection:continuity-Ttilde-general}

We discuss here the continuity of the operator in~\eqref{def:Ttilde-general}.
We have the following result.

\begin{theorem} \label{theorem:continuity-Ttilde-general}
Let~$g$, $\Ttilde$, and~$\jbf$
be as in~\eqref{options-varphi}, \eqref{def:Ttilde-general}, and~\eqref{jbf-fixed}.
Assume that~$g$ belongs to~$H^{d-1}_0(\Omega)$.
Given $1\le p < n/(n-1)$
and~$p'$ the conjugate index of~$p$,
the following inequality holds true:
\footnotesize{\[
\begin{split}
& \Norm{\Ttilde g}_{0,\Omega} \\
& \le 2^{\frac{n-1}{2}} \!\!\!\!\!\!\!\!\!\!\!\!\!\!\!
       \sum_{\kbf,\ellbf \in \{ 1,\dots,d \}^n ,\kbf+\ellbf=(j_1-1,\dots,j_d) }
            \left[
            \left(  \rho^{-1} \Norm{\partial^{\SemiNorm{\kbf}}_{x_1^{k_1-1},x_2^{k_2} \dots,x_n^{k_n}} \varphi}_{L^1(\Rbbn)}
            \!\!\!\!\!\!\!\! +
            \rho \Norm{\partial^{\SemiNorm{\kbf}+2}_{x_1^{k_1+1},x_2^{k_2},\dots,x_n^{k_n}} \varphi}_{L^1(\Rbbn)} \right)
            \Norm{\partial^{\SemiNorm{\ellbf}}_{x_1^{\ell_1},\dots,x_n^{\ell_n}} g}_{0,\Rbbn} \right] \\
&\quad + \frac{2^{\frac{n}{2}}}{(1-\frac{n}{p'})^{\frac{p}{2}}}
  \SemiNorm{\Omega}^{1-\frac p2}
  \Norm{\partialxjoxjn \varphi}_{L^1(\Omega)}^{\frac{p}{2}}
  \Norm{\partialxjoxjn \varphi}_{L^\infty(\Omega)}^{1-\frac{p}{2}}
  \Norm{g}_{0,\Omega}.
\end{split}
\]}\normalsize
\end{theorem}
\begin{proof}
The assertion follows combining
Lemmas~\ref{proposition:Ttildeo-general} and~\ref{proposition:Ttildet-general},
and the explicit representation of the constants
$\Cphirhokbf$ in~\eqref{2.3Duran-general}.
\end{proof} 

Theorem~\ref{theorem:general-order-Babu-Aziz} follows using
Theorem~\ref{theorem:continuity-Ttilde-general},
the chain rule, and proceeding as in the proof
of Theorem~\ref{theorem:1st-Babu-Aziz}.
\medskip 

An arbitrary order Ne\v cas-Lions,
generalising the first order version in~\eqref{1st-Necas-Lions},
may be shown based on Theorem~\ref{theorem:general-order-Babu-Aziz}
following the proof of Proposition~\ref{proposition:1st-Necas-Lions};
for $d$ in $\Nbb$, it reads
\[
\Norm{f}_{H^{-(d-1)}(\Omega) / \Pbb_{d-1}(\Omega)}
\le \CNLd \Norm{\nabla f}_{-d,\Omega}
\qquad\qquad \forall f \in H^{-(d-1)}(\Omega) / \Pbb_{d-1}(\Omega),
\]
where $H^{-(d-1)}(\Omega) / \Pbb_{d-1}(\Omega)$
is the space~$H^{-(d-1)}(\Omega)$
equipped with the norm
\[
\Norm{f}_{H^{-(d-1)}(\Omega) / \Pbb_{d-1}(\Omega)}
:= \inf_{q_{d-1} \in \Pbb_{d-1}(\Omega)}
    \Norm{f- q_{d-1}}_{-(d-1),\Omega}.
\]
The constant~$\CNLd$ depends on the Babu\v ska-Aziz constants of all orders up to~$d$
and the Poincar\'e constant through~\cite{Verfurth:1999}.

\section{An alternative proof of the zero boundary conditions} \label{appendix:alternative-zero-trace}
We present here an alternative proof
of the properties discussed in Remark~\ref{remark:boundary-conditions},
as it is based on some technical results
that are stated in the literature,
the proof of which we were not able to find.
Such results may be useful in the derivation
of explicit estimates for the right-inverse of the divergence
in $W^{k,p}$ spaces;
see Remark~\ref{remark:nonhilbertian-Setting} below for additional comments on this point.

We begin by providing the reader with a detailed proof of
an alternative expression for the first derivatives of~$\ubf$,
which has been stated in \cite[Remark III.3.2]{Galdi:2011},
and then proceed along the same lines
as in \cite[Chapter~2]{Acosta-Duran:2017}.

We start by showing a preliminary result,
which requires the definition
of an operator
$\Gbftilde_{j}:\Omega\times\Omega\to \Rbb$ given by
\begin{equation} \label{def:Gbftilde}
\Gbftilde_{j}(x,y)
:=  \int_0^1 \frac{x-y}{t} \  \partial_{x_j}\omega
    \left( y + \frac{x-y}{t}  \right) \frac{\dt}{t^n}
    \qquad   \forall j =1,\dots, n.
\end{equation}
\begin{lemma}
Let~$\Gbf$ and~$\Gbftilde_{j}$,
$j =1,\dots, n$, be defined
as in~\eqref{eq:kernel_def} and~\eqref{def:Gbftilde}.
Then,
for all positive $\varepsilon$,
the following identity holds true:
\begin{equation}    \label{eq:rel_difG}
\partial_{x_j} \Gbf(x,y)
= - \partial_{y_j} \Gbf(x,y) + \Gbftilde_{j}(x,y)
\qquad\qquad \forall \SemiNorm{x-y}>\varepsilon.
\end{equation}
\end{lemma}
\begin{proof}
We fix two indices $j,k =1,\dots, n$
and show the assertion
on the $k$-th components of $\Gbf$ and~$\Gbftilde$.
The fact that~$\omega$
belongs to~$\Ccalzinf(\Brho)$
and direct computations reveal
\[
\partial_{x_j} \Gbf_k(x,y) 
= \int_0^1 \left[ \delta_{kj}\ \omega \left( y + \frac{x-y}{t}  \right)
  + \frac{(x-y)_k}{t} \  \partial_{x_j}\omega
  \left( y + \frac{x-y}{t}  \right) 
  \right]\frac{\dt}{t^{n+1}}
\]
and
\[
\begin{split}
\partial_{y_j} \Gbf_k(x,y)
&= \int_0^1 \left[ \frac{-\delta_{kj}}{t} \ \omega 
        \left( y + \frac{x-y}{t}  \right)
   +\frac{(x-y)_k}{t} \  \partial_{x_j}\omega
        \left( y + \frac{x-y}{t}  \right)
        \left(1-\frac1t \right) \right] \frac{\dt}{t^n} \\
&= - \int_0^1 \left[ \delta_{kj}\ \omega 
        \left( y + \frac{x-y}{t}  \right)
    +\frac{(x-y)_k}{t} \  \partial_{x_j}\omega 
        \left( y + \frac{x-y}{t}  \right)\right]
        \frac{\dt}{t^{n+1}}
    + (\Gbftilde_{j})_k(x,y).
\end{split}
\]
A combination of the two previous displays gives the assertion.
\end{proof}

Next, we prove an identity involving the first derivatives of~$\ubf$.
\begin{proposition}[Galdi's formula]\label{proposition:Galdi}
Let~$\Gbf$ and~$\Gbftilde_{j}$,
$j =1,\dots, n$, be defined
as in~\eqref{eq:kernel_def} and~\eqref{def:Gbftilde}.
Given~$f$ in~$H^1_0(\Omega)$,
consider~$\ubf$ as in~\eqref{eq:bogovskii}.
Then, the following identity holds true:
\begin{equation}\label{eq:galdi_form}
\partial_{x_j} \ubf (x) 
= \int_\Omega \Gbf(x,y) \ \partial_{y_j}f(y)\ \mathrm{d}y 
    + \int_\Omega \Gbftilde_{j}(x,y)\ f(y)\ \mathrm{d}y
    \qquad\forall  j =1,\dots, n.
\end{equation}
\end{proposition}
\begin{proof}
Without loss of generality, we assume
that $f$ belongs to $\Ccalzinf(\Omega)$;
the general assertion follows from a density argument.
Moreover, we prove the assertion on the $k$-th components
of $\ubf$, $\Gbf$, and~$\Gbftilde$.

For any~$\phi$ in~$\Ccalzinf(\Omega)$ and $j,k =1,\dots, n$, proceeding as in \cite[Lemma 2.3]{Acosta-Duran:2017},
we have
\[
\begin{split}       
-\int_\Omega
& \ubf_k(x)\ \partial_{x_j} \phi(x)\ \mathrm{d}x 
 = -\int_\Omega \int_\Omega \Gbf_k(x,y) f(y)\
    \partial_{x_j} \phi(x)\  \mathrm{d}x\ \mathrm{d}y \\
& = -\int_\Omega f(y)\ \lim_{\varepsilon\to 0}
        \left(\int_{|y-x|>\varepsilon} \Gbf_k(x,y) \ 
        \partial_{x_j} \phi(x)\  \mathrm{d}x\right) \mathrm{d}y \\
& = \int_\Omega f(y)\ \lim_{\varepsilon\to 0}
        \left(\int_{|y-x|>\varepsilon} \partial_{x_j}\Gbf_k(x,y)\
        \phi(x)\  \mathrm{d}x 
        -\int_{|y-\xi|=\varepsilon} \Gbf_k(\xi,y)\ \phi(\xi)
        \frac{y_j-\xi_j}{|y_j-\xi_j|} \mathrm{d}\xi\right) \mathrm{d}y.
\end{split}
\]
Using \eqref{eq:rel_difG} and switching
the order of integration lead us to
\[
\begin{split}       
-\int_\Omega & \ubf_k(x)\
 \partial_{x_j} \phi(x)\ \mathrm{d}x 
    = -\lim_{\varepsilon\to0} \int_\Omega \phi(x) 
    \left(\int_{|y-x|>\varepsilon}  \partial_{y_j}\Gbf_k(x,y)\ 
        f(y)\ \mathrm{d}y\right) \mathrm{d}x \\
& + \int_\Omega \int_\Omega (\Gbftilde_{j})_k(x,y)\ 
        f(y)\ \phi(x)\ \mathrm{d}x\  \mathrm{d}y
    -\lim_{\varepsilon\to0} \int_\Omega\int_{|y-\xi|=\varepsilon} 
        \Gbf_k(\xi,y)\ \phi(\xi)\ f(y) 
        \frac{y_j-\xi_j}{|y_j-\xi_j|} \mathrm{d}\xi\ \mathrm{d}y \\
& =: \mathcal{I}_1 + \mathcal{I}_2 + \mathcal{I}_3.
\end{split}
\]
Integrating by parts with respect to the~$y$ variable
and recalling that~$f$ belongs to~$\Ccalzinf(\Omega)$,
we obtain
\[
\begin{aligned} 
\mathcal{I}_1
& = \lim_{\varepsilon\to0}\int_\Omega \phi(x)
        \left(\int_{|x-y|>\varepsilon} \Gbf_k(x,y)\ \partial_{y_j}f(y)\ \mathrm{d}y
    -  \int_{|x-\zeta|=\varepsilon} \Gbf_k(x,\zeta)\ f(\zeta)
        \frac{x_j-\zeta_j}{|x_j-\zeta_j|} \mathrm{d}\zeta \right)\mathrm{d}x \\
& = \int_\Omega \int_\Omega \Gbf_k(x,y)\ 
    \partial_{y_j}f(y)\ \phi(x)\ \mathrm{d}y \ \mathrm{d}x 
    + \lim_{\varepsilon\to0}\int_\Omega \int_{|\zeta-x|=\varepsilon} 
        \Gbf_k(x,\zeta)\ f(\zeta)\ \phi(x)
        \frac{\zeta_j-x_j}{|\zeta_j-x_j|}
        \mathrm{d}x\ \mathrm{d}\zeta \\
& = \int_\Omega\int_\Omega \Gbf_k(x,y)\ \partial_{y_j}f(y)\ 
    \phi(x)\ \mathrm{d}y \ \mathrm{d}x - \mathcal{I}_3.
\end{aligned}
\]
Combining the two above displays, we infer
\[
-\int_\Omega \ubf_k(x)\ \partial_{x_j} \phi(x)\ \mathrm{d}x =
\int_\Omega\left(\int_\Omega \Gbf_k(x,y)\ \partial_{y_j}f(y)\ \mathrm{d}y +
\int_\Omega (\Gbftilde_{j})_k(x,y)\ f(y)\ \mathrm{d}y\right) \phi(x)\ \mathrm{d}x,
\]
which gives~\eqref{eq:galdi_form} for any~$f$
in~$\Ccalzinf(\Omega)$.
\end{proof}

We are in a position to prove that~$\ubf$
satisfies homogeneous boundary conditions.

\begin{corollary}
Given~$f$ in~$H^1_0(\Omega)\cap L^2_0(\Omega)$,
consider~$\ubf$ as in~\eqref{eq:bogovskii}.
Then, $\ubf$ belongs to~$[H^2_0(\Omega)]^n$.
\end{corollary}
\begin{proof}
\textbf{Step~1: a decomposition of~$\gradbf \ubf$.}
Proposition~\ref{proposition:Galdi} allows us to
express the first derivatives of~$\ubf$
as a sum of two contributions.
The first correspond to Bogovskiĭ's constructions
applied to the first derivatives of~$f$;
the second are equivalent to~\eqref{eq:Bogo_rightinv}
with~$\omega$ in~\eqref{eq:kernel_def}
replaced by its first derivatives,
which still belongs to~$\Ccalzinf(\Omega)$.

In other words, we write $\gradbf\ubf = \taubold + \etabold$,
where the $j$-th columns, $j =1, \dots, n$, of~$\taubold$
and~$\etabold$ are given by
\[
\taubold_j(x) =\int_\Omega \Gbf(x,y) \ \partial_{y_j}f(y)\ \mathrm{d}y ;
\qquad \qquad
\etabold_j(x) = \int_\Omega \Gbftilde_{j}(x,y)\ f(y)\ \mathrm{d}y.
\]
\noindent \textbf{Step~2: treating~$\taubold$.}
Let $j =1,\dots, n$ be fixed.
Consider a sequence $g_m$ in $L^\infty(\Omega)$ such that
$g_m\to \partial_{y_j}f$ in $L^2(\Omega)$ as $m\to\infty$.
We consider a sequence of tensors~$\taubold_m$
whose $j$-th columns are given by
\[
(\taubold_m)_j (x) = \int_\Omega \Gbf(x,y) \ g_m (y)\ \mathrm{d}y.
\]
Applying \cite[Proposition~$2.1$]{Acosta-Duran:2017},
it follows that~$(\taubold_m)_j$ is continuous
and vanishes on~$\partial\Omega$;
Aa a result of~\cite{Swanson-Ziemer:1999},
we obtain that~$(\taubold_m)_j$
belongs to~$[W^{1,\infty}_0(\Omega)]^n$.
Using \cite[Corollary~19, part (iii)]{Guzman-Salgado:2021}
\footnote{In the literature, this result is referred to
as generalised Poincar\'e inequality,
which differs from~\eqref{0th-Babu-Aziz}
as no zero average condition and zero boundary conditions
on~$f$ are imposed.}
applied to $\partial_{y_j}f-g_m\in L^2(\Omega)$,
we get $(\taubold_m)_j \to \taubold_j$ in $[H^1(\Omega)]^n$.
Since $(\taubold_m)_j$ is in $[W^{1,\infty}_0(\Omega)]^n$ for all $m$,
we deduce that $\taubold_j$ belongs to $[H^1_0(\Omega)]^n$.
\medskip

\noindent \textbf{Step~3: treating~$\etabold$ and conclusion.}
The proof that~$\etabold_j$
belongs to~$[H^1_0(\Omega)]^n$
essentially boils down
to the proof that~$\ubf$ in~\eqref{0th-Babu-Aziz}
vanishes on the boundary of~$\Omega$;
the only difference resides in the presence
of~$\Gbftilde$ in lieu of~$\Gbf$,
which solely impacts the constant in~\eqref{0th-Babu-Aziz}.
We conclude that~$\gradbf\ubf$ belongs
to~$[H^1_0(\Omega)]^{n\times n}$.
In other words, $\ubf$ belongs to~$[H^2_0(\Omega)]^n$,
since~\eqref{0th-Babu-Aziz} already gives
that $\ubf$ is in~$[H^1_0(\Omega)]^n$.
\end{proof}

\begin{remark} \label{remark:nonhilbertian-Setting}
Identity~\eqref{eq:galdi_form} may allow us
to prove a first order Babu\v ska-Aziz inequality
also in a non-Hilbertian setting,
namely we may substitute $H^k$-type spaces
by $W^{k,p}$-type spaces, $p \ne 2$.
However, this would come at the price of suboptimal estimates
as those in~\eqref{suboptimal-Galdi}.
The reason for this is the use of the Calder\'on-Zygmund theory
instead of the Fourier transform approach by Dur\'an
while handling the term~$\Ttildeo$
in Section~\ref{subsection:continuity-Ttildeo}.
\end{remark}

\end{document}